\documentclass[11pt]{article}

\usepackage[margin=1in]{geometry}
\usepackage[utf8]{inputenc}
\usepackage[T1]{fontenc}
\usepackage{xcolor}
\usepackage{mathtools}
\usepackage{amsthm}
\usepackage{amsfonts}
\usepackage{algorithm}
\usepackage{algorithmic}
\usepackage{comment}
\usepackage{graphicx}
\usepackage{booktabs}
\usepackage{nicefrac}
\usepackage{microtype}
\usepackage{url}
\usepackage{hyperref}
\usepackage{bbm}
\usepackage[style=nature]{biblatex}

\hypersetup{
  colorlinks,
  linkcolor={red!40!gray},
  citecolor={blue!40!gray},
  urlcolor={blue!70!gray}
}

\newcommand{\e}{\varepsilon}
\DeclareMathOperator{\nnz}{nnz}
\DeclareMathOperator{\rank}{rank}
\DeclareMathOperator*{\tr}{tr}

\DeclareMathOperator*{\diag}{diag}

\newtheorem{theorem}{Theorem}[section]

\newtheorem{lemma}[theorem]{Lemma}

\theoremstyle{remark}
\newtheorem{remark}[theorem]{Remark}

\theoremstyle{definition}
\newtheorem{definition}[theorem]{Definition}

\newcommand{\bS}{\mathbf{S}}
\newcommand{\bU}{\mathbf{U}}
\newcommand{\bV}{\mathbf{V}}
\newcommand{\bI}{\mathbf{I}}
\newcommand{\bX}{\mathbf{X}}
\newcommand{\bY}{\mathbf{Y}}
\newcommand{\bM}{\mathbf{M}}
\newcommand{\bx}{\mathbf{x}}
\newcommand{\by}{\mathbf{y}}
\newcommand{\bu}{\mathbf{u}}
\newcommand{\bv}{\mathbf{v}}

\newcommand{\bD}{\mathbf{D}}

\newcommand{\bA}{\mathbf{A}}
\newcommand{\bP}{\mathbf{P}}
\newcommand{\bB}{\mathbf{B}}
\newcommand{\bBt}{\widetilde{\mathbf{B}}}

\newcommand{\bXt}{\widetilde{\mathbf{X}}}

\newcommand{\bC}{\mathbf{C}}

\newcommand{\bQ}{\mathbf{Q}}
\newcommand{\bW}{\mathbf{W}}
\newcommand{\bK}{\mathbf{K}}
\newcommand{\bSigma}{\mathbf{\Sigma}}
\newcommand{\bOmega}{\mathbf{\Omega}}
\newcommand{\orth}{{\mathrm{orth}}}

\DeclarePairedDelimiter{\norm}{\lVert}{\rVert}

\DeclarePairedDelimiter{\ip}{\langle}{\rangle}
\DeclarePairedDelimiter{\paren}{(}{)}
\DeclarePairedDelimiter{\sqbr}{[}{]}

\newcommand{\E}{\mathbb{E}}
\newcommand{\R}{\mathbb{R}}

\title{Accelerating Power Method with Fast Sketching\\ for Stronger Low-Rank Approximation}

\author{
  Shabarish Chenakkod\thanks{University of Michigan, Ann Arbor, MI 48105. Email: \texttt{shabari@umich.edu}.}
  \and
  Michał Dereziński\thanks{University of Michigan, Ann Arbor, MI 48105. Email: \texttt{derezin@umich.edu}.}
}

\date{}

\begin{document}

\maketitle

\begin{abstract}
The power method is one of the most fundamental tools for extracting top principal components from data through low-rank matrix approximation. Yet, when the target rank is large, the cost of matrix multiplication associated with this procedure becomes a major bottleneck. We develop an algorithmic and theoretical framework for accelerating the power method using fast sketching, which is a popular paradigm in randomized linear algebra. Our framework leads to simple and provably efficient methods for singular value decomposition, low-rank factorization, and Nystr\"om approximation, which attain strong numerical performance on benchmark problems. The key novelty in our analysis is the use of regularized spectral approximation, a property of fast sketching methods which proves more flexible in generalizing power method guarantees than traditional arguments.
\end{abstract}

\section{Introduction}

Computing top singular values and singular vectors of a large matrix $\bA\in\R^{m\times n}$ is one of the most fundamental tasks in data science. A simple yet powerful approach for doing this is the \emph{power method}. In its most basic form, the power method starts from a random (say, Gaussian) vector $\boldsymbol{\omega}\in\R^n$. Then, it repeatedly applies matrix $\bA$ and its transpose, obtaining a new vector $\by = (\bA\bA^\top)^q\bA\boldsymbol{\omega}$. As $q$ grows large, this vector becomes more and more closely aligned with the top singular vector of $\bA$. A natural extension of this strategy when seeking top $k$ singular vectors is to start from a random Gaussian matrix $\bOmega\in\R^{n\times O(k)}$, and then compute an orthonormal basis for the columns of a powered-up matrix $\bY = (\bA\bA^\top)^q\bA\bOmega$. The resulting matrix $\bQ = \mathrm{orth}(\bY)$, known as the \emph{range finder}, defines a subspace that is closely aligned with the span of the top $k$ singular vectors of $\bA$. This simple procedure offers strong theoretical guarantees for its approximation quality and lies at the core of widely used algorithms such as Randomized SVD \cite{halko2011finding} as well as many other low-rank factorization techniques \cite{rokhlin2010randomized,gu2015subspace,tropp2023randomized}. However, as $k$ gets large, the matrix multiplication cost associated with repeatedly applying $\bA$ becomes a major bottleneck of this approach.

Another powerful tool for low-rank approximation is \emph{fast sketching}, which is widely used to avoid expensive matrix multiplication in dimensionality reduction algorithms \cite{sarlos2006improved,clarkson2017low,woodruff2014sketching}. For this approach, instead of a dense Gaussian matrix $\bOmega$, we use a carefully constructed random matrix $\bS\in\R^{n\times O(k)}$ to compute a so-called sketch $\tilde\bA = \bA\bS$ of the matrix $\bA$, where the associated matrix product can be computed almost as quickly as it takes to read matrix $\bA$, thanks to the structure imposed on $\bS$ (e.g., because $\bS$ is an extremely sparse random matrix \cite{clarkson2017low}). Since $\bA\bS$ can be viewed as a rough version of the power method initialized from $\bS$ (by setting $q=0$), fast sketching has been used to construct range finders and low-rank factorizations under highly constrained computational budgets. A seemingly natural refinement of this would be to combine fast sketching with the power method by computing $(\bA\bA^\top)^q\bA\bS$ for some $q\geq 1$, however unfortunately, subsequent applications of $\bA$ no longer take advantage of fast sketching, thereby largely erasing the benefit of this approach \cite{tropp2023randomized}.

Another natural strategy for accelerating the power method with fast sketching is to replace each application of $\bA$ with an application of a sketch $\tilde\bA$ (while still initializing with a Gaussian matrix $\bOmega$). This leads to the following \emph{sketched power method} for constructing a range finder:
\begin{align}
    \bQ := \mathrm{orth}\big( (\tilde\bA\tilde\bA^\top)^q\tilde\bA \bOmega\big),\qquad\text{where}\quad \tilde\bA = \bA\bS\in\R^{m\times O(l)},\quad\bOmega\in\R^{n\times O(k)}, \quad l\gg k.\label{eq:sketched-power}
\end{align}
This strategy has the potential to be effective as long as the sketch $\tilde \bA$ is sufficiently larger than $\bOmega$ (i.e., $l\gg k$), so that it carries enough information about the spectrum of $\bA$, but also sufficiently smaller than $\bA$ (i.e., $l\ll n$), to reduce the computational cost of repeated matrix multiplications. Crucially, the cost of fast sketching methods is not affected by choosing a very large sketch size, so we can attain this regime essentially for free. The intuition is that, while the sketched power method no longer converges to a perfect alignment with the top-$k$ subspace of $\bA$ as $q$ grows large, it may reach a sufficiently good approximation much faster than the classical power method due to smaller per-iteration cost. Surprisingly, despite the simplicity of this procedure and 
significant prior work on low-rank approximation via both fast sketching \cite{sarlos2006improved,clarkson2017low,woodruff2014sketching,gittens2013revisiting,cohen_nelson_woodruff,nakatsukasa2020fast} and the power method \cite{rokhlin2010randomized,hardt2014noisy,gu2015subspace,tropp2023randomized,chang2026improving,mitliagkas2013memory}, little is known about the effectiveness of sketched power method in low-rank~approximation. See Appendix~\ref{a:related-work} for a detailed discussion of related work.

\subsection{Our contributions}
In this work, we develop a theoretical framework for analyzing the effectiveness of the sketched power method under a variety of fast sketching algorithms, and show how it can be used to improve the efficiency of standard tools for top-$k$ SVD and low-rank factorization. 
\begin{enumerate}
    \item \textbf{Sketch-Powered Range Finder}. We show that  \eqref{eq:sketched-power} yields an effective algorithm for constructing a range finder $\bQ$ whose approximation error is not much larger than if it was computed directly as $\bQ=\orth(\bA\bS)$. Thus, as long as the sketch size is sufficiently larger than the target rank (i.e., $l\gg k$), fast sketching is guaranteed to make the first several steps of the power method more efficient at attaining a strong approximation. 
    \item \textbf{Sketch-Powered Low-Rank Factorization}. The simplest strategy for using the range finder to obtain a low-rank factorization is the Randomized SVD algorithm \cite{halko2011finding}. However, this approach still requires one additional expensive matrix product. We show how to avoid this bottleneck by incorporating \eqref{eq:sketched-power} into another commonly used low-rank factorization technique (generalized Nystr\"om, \cite{nakatsukasa2020fast}), and analyze the resulting approximation guarantee.
    \item \textbf{Sketch-Powered Nystr\"om Approximation}. We further optimize our approach for the task of factorizing positive semidefinite matrices. By incorporating sketched power method into the classical Nystr\"om approximation, we obtain additional computational gains for this task.
\end{enumerate}
In particular, our results yield the following simple and remarkable conclusion (Theorem \ref{thm:generalised_nystrom}): 
\begin{center}
    \emph{Given $\bA\in\R^{n\times n}$ and $k\le O(\sqrt n)$, sketch-powered low-rank factorization takes $\tilde O(n^2)$ time\\ to compute a rank $O(k)$ factorization $\widehat\bA=\bY\bX$ such that $\|\bA - \widehat\bA\|\leq \sqrt{k+1}\cdot \sigma_{k+1}(\bA)$.}
\end{center}
Without sketched power method, the commonly used low-rank factorization schemes either only attain a weaker $\|\bA - \widehat\bA\|= O(\sqrt{n/k})\cdot \sigma_{k+1}(\bA)$ approximation guarantee in the same time budget, or they invoke the classical power method which leads to a higher $\tilde O(n^2k)$ time complexity. We confirm these computational gains empirically by comparing Randomized SVD with and without sketched power method on large matrices from benchmark datasets, as well as similarly evaluating our sketch-powered low-rank factorization~algorithm.

\subsection{Overview of the techniques}
Classical analysis of the power method and range finder construction typically follows one of two common strategies: Either it relies on the special rotational invariance properties of the Gaussian distribution \cite{halko2011finding}, or it builds off of the \emph{subspace embedding} property \cite{cohen_nelson_woodruff}, which is one of the corner-stones of fast sketching analysis \cite{woodruff2014sketching}.  However, neither of these approaches appears directly suitable to the sketched power method, perhaps owing to the fact that it combines two different sources of randomness: a structured sketching matrix $\bS$ and a dense Gaussian initialization $\bOmega$.

We depart from the classical analysis, and instead build on a different property of fast sketching called \emph{regularized spectral approximation}, which holds when a sketch $\tilde\bA=\bA\bS$ satisfies $\tilde\bA\tilde\bA^\top\!+\lambda\bI \approx_{\e} \bA\bA^\top\!+\lambda\bI$ for a suitably chosen regularization parameter $\lambda\geq 0$ and accuracy $\e\in[0,1]$. This property has been used in prior works primarily in the context of sketching for kernel ridge regression or regularized least squares \cite{alaoui2015fast,musco2017recursive,rudi2018fast,garg2025faster}, however it seems less obviously suitable for the analysis of the power method, since there is no explicit regularization occurring in the algorithm.

We first show that a simple argument based on the regularized spectral approximation property can be used to recover classical approximation guarantees for the range finder and the power method, when initialized not only with a Gaussian starting matrix, but also with many fast sketching matrices (Lemma \ref{lem:projspecapprox}). We then extend this argument to the sketched power method, by showing that a subspace approximation guarantee with respect~to~$\tilde\bA$ can be translated to a guarantee with respect~to~$\bA$, up to an additive error that depends only on the regularized spectral approximation quality of $\tilde\bA$ (Theorem~\ref{t:main-technical}).
Finally, in order to turn this approach into a proper low-rank factorization, we further extend the analysis to the generalized Nystr\"om approximation by introducing another fast sketching matrix and  building on sketching guarantees for generalized regression \cite{cohen_nelson_woodruff}.

\paragraph{Notation.} For a matrix $\bA\in\R^{m\times n}$ we use $\|\bA\|$ and $\|\bA\|_F$ to denote the spectral and Frobenius norm, respectively, with $\sigma_i(\bA)$ denoting its $i$th largest singular value and $\bA^\dagger$ denoting the Moore-Penrose pseudoinverse. For a positive semidefinite (psd) matrix $\bM$, we use $\lambda_i(\bM)$ to denote its $i$th largest eigenvalue, and we use $\bM\preceq\mathbf{N}$ to denote the Loewner ordering between symmetric matrices. 
We let $\mathrm{orth}(\bA)$ denote the matrix with orthonormal columns spanning the range of $\bA$.

\section{Main Results}
\label{s:main}
In this section, we incorporate sketched power method into three low-rank approximation algorithms, and provide theoretical guarantees for the running time and approximation error of these methods.

\paragraph{Sketch-Powered Range Finder.} We start with the most basic version of our approach, which is constructing the range finder $\bQ$ as described in Section \ref{alg:column_sketched_iteration}. This is formally given in Algorithm~\ref{alg:column_sketched_iteration}.
\begin{algorithm}[H] 
\caption{Sketch-Powered Range Finder}
\begin{flushleft}
{\bf inputs}:  $\bA \in \mathbb{R}^{m \times n}$, $\bS \in \mathbb{R}^{n \times r_1}$, $\bOmega \in \mathbb{R}^{r_1 \times r_2}$, $q \in \mathbb{N}$

{\bf output}: $\mathbf{Q} \in \mathbb{R}^{m \times n}$
\end{flushleft}
\vspace{-2mm}
\begin{algorithmic}[1]\label{alg:column_sketched_iteration}
\STATE{Compute $\mathbf{Y} := (\tilde\bA \tilde\bA^\top)^q \tilde\bA\bOmega$,\quad where $\tilde\bA = \bA\bS$.}
\STATE{Orthonormalize the columns of $\mathbf{Y}$ to obtain $\mathbf{Q} \in \R^{m\times r_2}$.}
\RETURN{$\bQ$}
\end{algorithmic}
\end{algorithm}
\paragraph{Comparison with Classical Range Finder.} If we let $\bS$ be the identity matrix, then this recovers the classical range finder construction. Here, for simplicity of presentation, we omit the intermediate re-orthonormalization step which is often performed in practice after each iteration of the power method for the sake of numerical stability (this does not affect the output in exact arithmetic). The guarantee for the classical range finder construction with block size $r_2 = O(k)$ can be stated as \cite{halko2011finding}:
\begin{align*}
    \|\bA - \bQ\bQ^\top\bA\|^2 \leq \Big(2\sigma_{k+1}^2(\bA) + \frac1k\sum_{i>k}\sigma_i^2(\bA)\Big)^{\frac1{2q+1}}\leq \Big(\frac{\min\{m,n\}}{k}\Big)^{\frac1{2q+1}}\sigma_{k+1}^2(\bA),
\end{align*}
where $\|\bA - \bQ\bQ^\top\bA\|^2$ measures how much of the spectrum of $\bA$ is captured in the  column span~of~$\bQ$. By choosing $q = \tilde O(1/\e)$, we can bound this error by $(1+\e)\sigma_{k+1}^2(\bA)$. However, when $k$ is large this requires expensive matrix multiplication which takes $\tilde O(\nnz(\bA)k/\e)$ time for $\bA$ with $\nnz(\bA)$ non-zero entries, or $\tilde O(mnk/\e)$ for a dense $\bA$. To avoid this cost, one can choose $q=0$ and replace $\bOmega$ with a fast sketching matrix, but then the squared error can be as large as $\frac {\min\{m,n\}}k\sigma_{k+1}^2(\bA)$.

\paragraph{Our result.} By introducing fast sketching directly into the power method, we can substantially optimize this trade-off. In particular, relying on sparse random sketching matrices $\bS$ \cite{clarkson2017low,nelson2013osnap,chenakkod2025optimal}, we obtain the following guarantee for Sketch-Powered Range Finder (see Section~\ref{s:analysis} for details).
\begin{theorem}\label{thm:column_sketched_iteration}
    Given matrix $\bA \in \mathbb{R}^{m \times n}$, $1\leq k\leq l\leq \min\{m,n\}$ and $\e\in(0,1/2]$, Algorithm~\ref{alg:column_sketched_iteration} implemented using a sparse matrix $\bS$ of size $r_1=\tilde O(l/\e^2)$, a Gaussian matrix $\bOmega$ of size $r_2 = O(k)$, and $q=\tilde O(1/\e)$, in time $\tilde O(\nnz(\bA)/\e + mlk/\e^3)$ computes $\bQ\in\R^{m\times O(k)}$  such that
        \begin{align*}
        \norm{\bA - \bQ\bQ^\top\bA}^2 &\le \paren*{1 + \e}  \sigma_{k+1}^2(\bA) + \frac{\varepsilon}{l} \sum_{i>l} \sigma^2_{i}(\bA).
    \end{align*}
\end{theorem}
\begin{remark} \label{rem:error}
    Due to the presence of the additive error term, it is often better to optimize this bound by choosing a large $l$, rather than a small $\e$.
    For example,  given $k\le O(\sqrt n)$, we can obtain $\|\bA-\bQ\bQ^\top\bA\|\leq \sqrt{k+1}\cdot\sigma_{k+1}(\bA)$ in time $\tilde O(mn)$ with $l = O(n/k)$ and $\e=\Theta(1)$. This is better than the $\sqrt{\min\{m,n\}/k}\cdot\sigma_{k+1}(\bA)$ bound attained by the classical range finder in the same~time.
\end{remark}
\paragraph{Sketch-Powered Low-Rank Factorization.} We can obtain a low-rank factorization from the range finder $\bQ$ by computing $\bB = \bQ^\top\bA$ and returning factors $\bQ$, $\bB$. For instance, this strategy is employed by the Randomized SVD algorithm \cite{halko2011finding} to compute $\bU\bSigma\bV^\top\!\approx\bA$ by performing the SVD of $\bB=\tilde \bU\bSigma\bV^\top$ and computing $\bU=\bQ\tilde\bU$, where $\bSigma$ is diagonal and $\tilde\bU,\bV$ are orthonormal. Thus, the sketched power method can effectively accelerate Randomized SVD, as shown numerically in Section \ref{s:experiments}. Yet, computing $\bB$ still requires one expensive matrix product which can be a bottleneck. 

We get around this issue by incorporating sketched power method into a slightly modified low-rank factorization scheme \cite{nakatsukasa2020fast}, which uses a second fast sketching matrix to effectively approximate the projection of $\bA$ onto the subspace defined by $\bQ$. The procedure is formalized in Algorithm \ref{alg:generalised_nystrom}.
\begin{algorithm}[H]
\caption{Sketch-Powered Low-Rank Factorization}
\begin{flushleft}
{\bf inputs}:  $\bA \in \mathbb{R}^{m \times n}$, $\bS_1 \in \mathbb{R}^{n \times r_1}$, $\bS_2\in \R^{m\times r_1}$, $\bOmega \in \mathbb{R}^{r_1 \times r_2}$ and $q \in \mathbb{N}$

{\bf output}: factors $\bY \in \mathbb{R}^{m \times r_2}$ and $\bX \in \mathbb{R}^{r_2 \times n}$ defining $\widehat\bA = \bY\bX$
\end{flushleft}
\vspace{-2mm}
\begin{algorithmic}[1]\label{alg:generalised_nystrom}
\STATE{Compute $\mathbf{Y} := (\tilde\bA \tilde\bA^\top)^q \tilde\bA\bOmega$,\quad where $\tilde\bA = \bA\bS_1$.}
\STATE{Compute $\bX := (\bS_2^\top \bY)^\dagger (\bS_2^\top \bA)$.}
\RETURN{$\bY,\bX$}
\end{algorithmic}
\end{algorithm}

\begin{theorem}\label{thm:generalised_nystrom}
Given matrix $\bA \in \mathbb{R}^{m \times n}$, $1\leq k\leq l\leq \min\{m,n\}$ and $\e\in(0,1/2]$, Algorithm~\ref{alg:generalised_nystrom} implemented using sparse matrices $\bS_1,\bS_2$ of size $r_1=\tilde O(l/\epsilon^2)$, a Gaussian matrix $\bOmega$ of size $r_2 = O(k)$, and $q=\tilde O(1/\e)$, in time $\tilde O(\nnz(\bA)/\e + mlk/\e^3 + nlk/\e^2)$ computes $\bY\in\R^{m\times O(k)}$ and $\bX\in\R^{O(k)\times n}$  such that
\begin{align*}
\norm{\bA - \bY\bX}^2 &\le \paren*{1 + \e}  \sigma_{k+1}^2(\bA) + \frac{\varepsilon}{l} \sum_{i>l} \sigma^2_{i}(\bA).
\end{align*}
\end{theorem}
\begin{remark}
    In particular, given an $n\times n$ matrix $\bA$ and $k\le O(\sqrt n)$,
    we can find $\bY,\bX^\top\in\R^{n\times O(k)}$ such that 
    $\|\bA-\bY\bX\|\leq \sqrt{k+1}\cdot \sigma_{k+1}(\bA)$ in time $\tilde O(n^2)$ by using $l = O(n/k)$ and $\e=\Theta(1)$.
\end{remark}

\paragraph{Sketch-Powered Nystr\"om Approximation.} When the matrix $\bA$ is symmetric positive semidefinite (psd), then it is natural to seek a low-rank factorization that is also psd. A standard approach for doing this is the so-called Nystr\"om approximation \cite{williams2000using,gittens2013revisiting,alaoui2015fast,musco2017recursive,rudi2018fast}, where we use a sketching matrix $\bS$ to construct factors $\bC=\bA\bS$ and $\bW=\bS^\top\bA\bS$ such that $\widehat\bA=\bC\bW^\dagger\bC^\top$ forms a low-rank approximation of $\bA$. In  Algorithm \ref{alg:nystrom}, we show how to incorporate sketched power method into this procedure by relying on a large intermediate Nystr\"om approximation $\tilde\bC,\tilde\bW$ to construct a smaller (but still accurate) one. An added computational advantage of this construction is that the power method is performed on the matrix $\tilde\bW$ which has both dimensions reduced rather than just one.
\begin{algorithm}[H]
\caption{Sketch-Powered Nystr\"om Approximation}
\begin{flushleft}
{\bf inputs}:  psd $\bA \in \mathbb{R}^{n \times n}$, $\bS \in \mathbb{R}^{n \times r_1}$, $\bOmega\in \mathbb{R}^{r_1 \times r_2}$, $q \in \mathbb{N}$

{\bf output}: factors $\mathbf{C} \in \mathbb{R}^{n \times r_2}$ and $\mathbf{W} \in \mathbb{R}^{r_2 \times r_2}$ defining $\widehat{\bA} = \mathbf{C}\mathbf{W}^\dagger\mathbf{C}^\top$
\end{flushleft}
\vspace{-2mm}
\begin{algorithmic}[1]\label{alg:nystrom}
\STATE{Compute $\tilde\bC := \bA\bS$ and $\tilde\bW := \bS^\top\bA\bS$.}
\STATE{Compute $\bY := \tilde\bW^q\bOmega$.}
\STATE{Compute $\mathbf{C} := \tilde\bC\bY$ and $\mathbf{W} := \bY^\top\tilde\bW\bY$.}
\RETURN{$\mathbf{C},\mathbf{W}$}
\end{algorithmic}
\end{algorithm}

\begin{theorem}\label{thm:sketched_power_nystrom}
Given a psd matrix $\bA \in \mathbb{R}^{n \times n}$, $1\leq k\leq l\leq n$ and $\e\in(0,1/2]$, Algorithm~\ref{alg:nystrom} implemented using a sparse matrix $\bS$ of size $r_1=\tilde O(l/\epsilon^2)$, a Gaussian matrix $\bOmega$ of size $r_2 = O(k)$, and $q=\tilde O(1/\e)$, in time $\tilde O(\nnz(\bA)/\e + nlk/\e^2 + l^2k/\e^5)$ computes $\bC\in\R^{n\times O(k)}$ and a psd matrix $\bW\in\R^{O(k)\times O(k)}$  such that
\begin{align*}
\norm{\bA - \bC\bW^\dagger\bC^\top} &\le \paren*{1 + \e}  \lambda_{k+1}(\bA) + \frac{\varepsilon}{l} \sum_{i>l} \lambda_{i}(\bA).
\end{align*}
\end{theorem}

\section{Analysis via Regularized Spectral Approximation}
\label{s:analysis}
Our analysis relies on the following approximation property for a sketch, which has been previously used for low-rank approximation in certain specialized settings such as kernel ridge regression \cite{alaoui2015fast,musco2017recursive,rudi2018fast}, but is not typically associated with guarantees for range finder construction.
\begin{definition}
 Let $\bA\in\R^{m\times n}$, $\lambda\geq 0$, and $\e\in[0,1/2]$. For a matrix $\bS \in \R^{n \times r}$, we say that $\bA\bS$ is a $\lambda$-regularized $\e$-spectral approximation of $\bA$ if
    \[
            \paren*{1-\varepsilon}\paren*{\bA \bA^\top + \lambda\bI}
            \preceq
            \bA \bS \bS^\top \bA^\top + \lambda\bI
            \preceq
            \paren*{1+\varepsilon}\paren*{\bA \bA^\top + \lambda\bI}.
    \]
\end{definition}
This property can be satisfied by most standard sketching matrices with an appropriate choice of dimension $r$ and regularizer $\lambda$, as shown in the following lemma which is an immediate consequence of \cite[Lemma 12]{derezinskisidford} combined with the discussion in \cite[Sections 2.1.1-2.1.3]{cohen_nelson_woodruff} (see Appendix~\ref{subsec:ammproofs}).
\begin{lemma}[\cite{cohen_nelson_woodruff,derezinskisidford}]\label{lem:specapprox}
    Given $\bA\in\R^{m\times n}$, $\e,\delta\in(0,1)$, and $1\leq k\leq \mathrm{rank}(\bA)$, consider:
    \begin{align}
        \lambda = \frac1k\sum_{i>k}\sigma_i^2(\bA).\label{eq:lambda}
    \end{align}
    The following distributions over a sketching matrix $\bS\in\R^{n\times r}$ have the property that $\bA\bS$ is a $\lambda$-regularized $\e$-spectral approximation of $\bA$ with probability $1-\delta$:
    \begin{enumerate}
        \item $\bS$ has independent sub-Gaussian entries scaled by $1/\sqrt r$, for example  $\bS_{ij}\sim\mathcal{N}(0,1/r)$ or $\bS_{ij}\sim\pm1/\sqrt r$, and the sketch size satisfies $r\geq O((k+\log(1/\delta))/\e^2)$.
        \item $\bS$ is a CountSketch-type matrix with $s$ random non-zero entries per row of the form $\pm1/\sqrt s$, and one of the following holds:\vspace{-2mm}
        \begin{enumerate}
            \item $r \geq O((k+\log(1/\e\delta))/\e^2)$ and $s\geq O(\log^2(k/\delta)/\e + \log^3(k/\delta))$;
            \item $r\geq O(k\log(k/\delta)/\e^2)$ and $s\geq O(\log(k/\delta)/\e)$.
        \end{enumerate}
        \item $\bS$ is a Subsampled Randomized Hadamard Transform, i.e., $\bS=\frac1{\sqrt r}\bD\mathbf{H}\bI_{:,S}$, where $\bD$ is diagonal with random $\pm1$ entries, $\mathbf{H}$ is a Hadamard matrix, and $S$ is a uniformly random subset of $\{1,...,n\}$ with size $r\geq O((k+\log(n/\delta))\log(k/\delta)/\e^2)$.
    \end{enumerate}
\end{lemma}
We start by showing how the classical guarantees for range finder construction can be easily recovered up to constant factors from the regularized spectral approximation property. We consider this result to be folklore, as similar arguments have appeared in prior work, although in different contexts \cite{derezinskisidford}.
\begin{lemma} \label{lem:projspecapprox}
    Given matrix $\bA\in\R^{m\times n}$ and $\bOmega\in\R^{n\times r}$, let $\bQ=\mathrm{orth}(\bA\bOmega)$. If $\bA\bOmega$ is a $\lambda$-regularized $1/2$-spectral approximation of $\bA$ for some $\lambda \geq 0$, then 
    \begin{align*}
        \|\bA - \bQ\bQ^\top\bA\|^2 \leq 2\lambda\quad\text{and}\quad \|\bA-\bQ\bQ^\top\bA\|_F^2\leq 2\cdot \min_{l}\Big\{  l\lambda + \sum_{i>l}\sigma_i^2(\bA)\Big\}.
    \end{align*}
    In particular, if $\lambda$ is as in \eqref{eq:lambda} for some $1\leq k\leq n$ (e.g., $\bOmega\in\R^{n\times O(k)}$ is a Gaussian matrix), then:
    \begin{align*}
        \|\bA - \bQ\bQ^\top\bA\|^2 \leq \frac2k\sum_{i>k}\sigma_i^2(\bA)\quad\text{and}\quad \|\bA-\bQ\bQ^\top\bA\|_F^2\leq 4\sum_{i>k}\sigma_i^2(\bA).        
    \end{align*}
\end{lemma}
\begin{remark}
    While the above statement corresponds only to the power method with $q=0$, one can easily extend it to $q>0$ by replacing $\bA$ with $(\bA\bA^\top)^q\bA$ and following standard reductions \cite{halko2011finding}.
\end{remark}
\begin{proof}
    Let $\bP_{\bA\bOmega}:=\bQ\bQ^\top$. Using the fact that $\bP_{\bA \bOmega} = \bA\bOmega\bOmega^\top\bA^\top(\bA\bOmega\bOmega^\top\bA^\top)^\dagger$, we have
    \begin{align*}
        \bI - \bP_{\bA \bOmega} &= (\bA\bOmega\bOmega^\top\bA^\top + \lambda\bI)(\bA\bOmega\bOmega^\top\bA^\top + \lambda\bI)^{-1} - \bP_{\bA \bOmega} \\
        &= \bA\bOmega\bOmega^\top\bA^\top(\bA\bOmega\bOmega^\top\bA^\top + \lambda\bI)^{-1} - \bP_{\bA \bOmega} + \lambda(\bA\bOmega\bOmega^\top\bA^\top + \lambda\bI)^{-1} \\
        &\preceq \lambda(\bA\bOmega\bOmega^\top\bA^\top + \lambda\bI)^{-1} 
        \preceq 2\lambda(\bA\bA^\top + \lambda\bI)^{-1}.
    \end{align*}
    where the second to last inequality follows from the fact that $\bA\bOmega\bOmega^\top\bA^\top(\bA\bOmega\bOmega^\top\bA^\top + \lambda\bI)^{-1} \preceq \bP_{\bA \bOmega}$ and the last inequality follows from the regularized spectral approximation guarantee. So,
    \begin{align*}
     \|\bA - \bP_{\bA\bOmega}\bA\|^2 &= \norm{\bA^\top \paren*{\bI - \bP_{\bA \bOmega}} \bA } 
     \le 2\lambda \norm{\bA^\top (\bA\bA^\top + \lambda\bI)^{-1} \bA } \le 2\lambda.
    \end{align*}
    Similarly, we get $\|\bA - \bP_{\bA\bOmega}\bA\|_F^2\leq 2\lambda\tr(\bA^\top (\bA\bA^\top + \lambda\bI)^{-1} \bA)$. Finally, for any $l\geq 0$, we have
    \begin{align*}
        \tr(\bA^\top (\bA\bA^\top + \lambda\bI)^{-1} \bA) = \sum_{i}\frac{\sigma_i^2(\bA)}{\sigma_i^2(\bA)+\lambda}\leq l + \frac1\lambda\sum_{i>l}\sigma_i^2(\bA),
    \end{align*}
    which concludes the proof.
\end{proof}
Our main technical result is essentially an extension of Lemma~\ref{lem:projspecapprox} to the sketched power method. Note that we provide bounds for both the spectral and Frobenius norm error, as both are needed later for Theorem \ref{thm:generalised_nystrom}. We provide the proof sketch for each bound in Sections \ref{subsec:specnormbd} and \ref{subsec:forbnorm}, respectively.
\begin{theorem}\label{t:main-technical}
    Given $\bA\in\R^{m\times n}$, $\bS\in\R^{n\times r_1}$, $\bOmega\in\R^{r_1\times r_2}$, and $q,\lambda_1,\lambda_2\geq 0$, suppose that:
    \begin{enumerate}
        \item $\bA\bS$ is a $\lambda_1$-regularized $\e$-spectral approximation of $\bA$ for some $\e\in[0,1/2]$;
        \item $\bB\bOmega$ is a $\lambda_2$-regularized $1/2$-spectral approximation of $\bB:=(\bA\bS(\bA\bS)^\top)^q\bA\bS$ .
    \end{enumerate}
    Then, $\bQ=\mathrm{orth}(\bB\bOmega)$ satisfies:
    \begin{align}
        \|\bA - \bQ\bQ^\top\bA\|^2 &\leq \paren*{1+2\e} \paren*{ \paren*{ 2\lambda_2 }^{\frac{1}{2q+1}}  + \varepsilon\lambda_1  }, \label{eq:main-technical-spec}\\
        \|\bA - \bQ\bQ^\top\bA\|_F^2 &\leq \min_r \Big\{ 8r\big(\lambda_2^{\frac{1}{2q+1}} + \lambda_1\big)  + \sum_{i>r} \sigma^2_{i}(\bA) \Big\}. \label{eq:main-technical-frob}
    \end{align}
\end{theorem}
Intuitively, by choosing matrix $\bS$ to be sufficiently larger than $\bOmega$, we can ensure that $\lambda_1\ll\lambda_2$, so that performing a few steps of power iteration reduces the overall error until $\lambda_2^{1/(2q+1)}\approx \lambda_1$. Following this intuition and selecting the sketching matrices according to Lemma~\ref{lem:specapprox} yields Theorem \ref{thm:column_sketched_iteration}.
\begin{proof}[Proof of Theorem \ref{thm:column_sketched_iteration}] 
Let $\lambda_1 = \frac1l\sum_{i>l}\sigma_i^2(\bA)$. By Lemma \ref{lem:specapprox}, we can generate an $n \times r_1$ sparse random matrix $\bS$ with $s = O(\log(l)/\e)$ non-zero entries per row such that with high probability $\tilde\bA := \bA\bS$ is a $\lambda_1$-regularized $\e$-spectral approximation of $\bA$ for $r_1 =  O(l\log(l)/\e^2)$, which yields
\begin{align}
    (1-\varepsilon)\paren*{ \sigma_i^2(\bA) + \lambda_1 } \le \sigma_i^2(\tilde\bA) &+ \lambda_1 \le (1+\varepsilon)\paren*{ \sigma_i^2(\bA) + \lambda_1 }. \label{eq:s1specapprox}
\end{align}
Let $\bB=(\tilde\bA\tilde\bA^\top)^q\tilde\bA$ and $\lambda_2 = \frac1k\sum_{i>k}\sigma_i^2(\bB)$. By Lemma \ref{lem:specapprox}, if $\bOmega$ is an $r_1 \times O(k)$ Gaussian matrix, then with high probability $\bB\bOmega$ is a $\lambda_2$-regularized $1/2$-spectral approximation of $\bB$.

To get the stated bound after applying Theorem \ref{t:main-technical}, we express $\paren*{ 2\lambda_2 }^{\frac{1}{(2q+1)}}$ in terms of $\sigma_i^2(\bA)$. Note that, $\sigma_i^2(\bB) = \sigma_i(\tilde\bA)^{2(2q+1)}  \le \paren*{(1+\varepsilon) \sigma_i(\bA)^2  + \lambda_1\varepsilon}^{2q+1}$ by \eqref{eq:s1specapprox}. Thus, $\paren*{ 2\lambda_2 }^{\frac{1}{(2q+1)}}$ is bounded by the $\ell_{2q+1}$ norm of a vector with coordinates $(1+\varepsilon) \sigma_i(\bA)^2  + \lambda_1\varepsilon$ for $i>k$. 
In Lemma~\ref{lem:lambda2bound} we show that for $q = \Omega(\log(m)/\e)$, this $\ell_{2q+1}$ norm is essentially the same as the largest coordinate:
\begin{align*}
    \paren*{ 2\lambda_2 }^{\frac{1}{(2q+1)}} \le (1+4\e)\sigma_{k+1}(\bA)^2 + 2\lambda_1\e.
\end{align*}
Then, by Theorem \ref{t:main-technical}, we have,
\begin{align*}
    \|\bA - \bQ\bQ^\top\bA\|^2 &\leq \paren*{1+2\e} \paren*{ \paren*{ 2\lambda_2 }^{\frac{1}{(2q+1)}}  + \varepsilon\lambda_1  } \\
    &\le (1+2\e) \paren*{(1+4\e)\sigma_{k+1}(\bA)^2 + 3\lambda_1\e } \\
    &\le (1+10\e)\sigma_{k+1}(\bA)^2 + 6\lambda_1\e.
\end{align*}
Replacing $\e$ by $\e/10$ recovers the bound. We can compute $\bA \bS$ in time $O(\nnz (\bA)\log(l)/\e)$, while $(\tilde\bA\tilde\bA^\top)^q\tilde\bA\bOmega$ takes $O(qmr_1r_2) = \tilde{O}(mlk/\e^3)$ time. Constructing $\mathbf{Q}$ takes $O(mk^2)$ time.
\end{proof}

Our psd low-rank factorization guarantee (Theorem \ref{thm:sketched_power_nystrom}) follows via a reduction from Theorem~\ref{thm:column_sketched_iteration}.
\begin{proof}[Proof of Theorem \ref{thm:sketched_power_nystrom}]
Let $\bQ$ denote the orthonormal basis returned by Algorithm \ref{alg:column_sketched_iteration} when applied to $\bA^{1/2}$ with sketches $\bS,\bOmega$. Then $\bQ$ is the orthonormalization of the columns of 
    \[
        \mathbf{K}
        :=
        \big(\bA^{1/2}\bS(\bA^{1/2}\bS)^\top\big)^{q}\bA^{1/2}\bS\bOmega = \bA^{1/2}\bS (\bS^\top \bA \bS)^q \bOmega = \bA^{1/2}\bS \bY.
    \]
    Thus $\bK^{\top} \bK = \bY^\top\tilde\bW\bY$  and therefore,
    \begin{align*}        
        \mathbf{C}\bW^\dagger \mathbf{C}^\top 
        &=
        \bA\bS\bY
        \big(\bY^\top\tilde\bW\bY\big)^\dagger
        \bY^\top\bS^\top\bA =
        \bA^{1/2}\mathbf{K}\paren*{\mathbf{K}^\top\mathbf{K}}^\dagger\mathbf{K}^\top\bA^{1/2}
        =
        \bA^{1/2}\bQ\bQ^\top\bA^{1/2}.
    \end{align*}
    Applying Theorem \ref{thm:column_sketched_iteration} with $\bA^{1/2}$ in place of $\bA$ gives
    \[
        \norm{\bA-\bC\bW^\dagger\bC^\top} = \norm{\paren*{\bI-\bQ\bQ^\top}\bA^{1/2}}^2
        \le
        \paren*{1+\e}\sigma_{k+1}^2(\bA^{1/2})
        +
        \frac{\e}{l}\sum_{i>l}\sigma_i^2(\bA^{1/2}).
    \]
    Since $\bA$ is psd, $\sigma_i^2(\bA^{1/2})=\sigma_i(\bA) = \lambda_i(\bA)$ for every $i$, which yields the bound. Since $\bS$ is sparse, the matrices $\tilde\bC=\bA\bS$ and $\tilde\bW=\bS^\top\tilde\bC$ can be computed in $\tilde{O}\paren*{\nnz(\bA)/\e}$ and $\tilde{O}\paren*{ln/\e^3}$ time, respectively. Computing $\bY=\tilde\bW^q\bOmega$ takes $\tilde{O}\paren*{kl^2/\e^5}$ time, and forming $\mathbf{C}=\tilde\bC\bY$ takes $O\paren*{kln/\e^2}$ time. Computing $\mathbf{W}=\bY^\top\tilde\bW\bY$ takes $O\paren*{kl^2/\e^4+k^2l/\e^2}$ time.
\end{proof}
Obtaining our general low-rank factorization guarantee (Theorem \ref{thm:generalised_nystrom}) requires additional effort, as it does not follow from Theorem \ref{thm:column_sketched_iteration}. Here, we rely on both the spectral and Frobenius norm bounds in Theorem \ref{t:main-technical}, as well as a sketching guarantee for generalized regression (see Appendix~\ref{subsec:ammproofs}).
\begin{lemma}[Generalized regression \cite{cohen_nelson_woodruff}] \label{lem:genreg}
Given $\bA\in\R^{n\times d},\bB\in\R^{n\times p}$, consider the generalized regression problem of finding $\bX$ to minimize $\norm{\bA\bX - \bB}$. Let $\bS\in\R^{n\times r}$ be one of the matrices from Lemma \ref{lem:specapprox} for some $k\ge d$. Then,  $\bXt :=(\bS^\top \bA)^\dagger \bS^\top\bB$ with probability $1-\delta$ satisfies, 
    \[
        \norm{\bA\bXt-\bB}^2
        \le (1+\e^2)\norm{\bP_{\bA}\bB-\bB}^2 + \frac{\e^2}{k}\norm{\bP_{\bA}\bB-\bB}_F^2.
    \]
\end{lemma}

\begin{proof}[Proof of Theorem \ref{thm:generalised_nystrom}]
    Let $\bB := (\tilde\bA \tilde\bA^\top)^{q}\tilde\bA$, where $\tilde\bA=\bA\bS_1$, and recall that $\bY = \bB\bOmega$. Let $\bS_1$ have $s = O(\log(l)/\e)$ non-zero entries per row with $r_1 =  O(l\log(l)/\e^2)$, and let $\bOmega$ be Gaussian with $r_2 = O(k)$. Then, by Theorem \ref{t:main-technical} (via Lemma \ref{lem:specapprox}), we get,
    \begin{align}
        \norm{\bA-\bP_{\bY}\bA}^2 &\le \paren*{1+2\e} \paren*{ \paren*{ 2\lambda_2 }^{\frac{1}{2q+1}}  + \varepsilon\lambda_1  },  \label{eq:gennys_spec}\\
        \norm{\bA-\bP_{\bY}\bA}^2_F &\le \min_r \Big\{ 8r(\lambda_2^{\frac1{2q+1}} + \lambda_1)  + \sum_{i>r} \sigma^2_{i}(\bA) \Big\}. \label{eq:gennys_frob}
    \end{align}
    where $\lambda_1 = \frac1l\sum_{i>l}\sigma_i^2(\bA)$ and $\lambda_2 = \frac1k\sum_{i>k}\sigma_i^2(\bB)$ as in the proof of Theorem \ref{thm:column_sketched_iteration}. 
    Let $\bS_2\in\R^{m\times r_1}$ have the same sparsity and sketching dimension as $\bS_1$.
    Then, by Lemma \ref{lem:genreg}, applied to the generalized regression problem $\min_{\bX} \norm{\bY \bX - \bA}$ with $\bX:=\paren*{\bS_2^\top\bY}^\dagger \bS_2^\top\bA$,
    \[
        \norm{\bY\bX-\bA}^2
        \le (1+\e^2)\norm{\bP_{\bY}\bA-\bA}^2 + \frac{\e^2}{l}\norm{\bP_{\bY}\bA-\bA}_F^2.
    \]
    Using \eqref{eq:gennys_spec}, \eqref{eq:gennys_frob}, and the bound $\paren*{ 2\lambda_2 }^{\frac{1}{2q+1}} \le (1+4\e)\sigma_{k+1}(\bA)^2 + 2\lambda_1\e$ from Lemma~\ref{lem:lambda2bound},
    \begin{align*}
        \norm{\bY\bX-\bA}^2
        &\le (1+\e^2)\paren*{ (1+10\e)\sigma_{k+1}(\bA)^2 + 6\lambda_1\e }  + \frac{\e^2}{l}\paren*{ 8l(\lambda_2^{1/2q+1} + \lambda_1)  + \sum_{i>l} \sigma^2_{i}(\bA)  } \\
        &\le \paren*{1 + 40\e}\sigma_{k+1}(\bA)^{2} + \frac{29\e}{l} \sum_{i>l} \sigma^2_{i}(\bA).
    \end{align*}
    Replacing $\e$ by $\e/40$, we are done. The matrix $\bY$ can be computed in time $\tilde{O}\paren*{\nnz(\bA) + mlk/\e^3}$, whereas forming $\bX=\paren*{\bS_2^\top\bY}^\dagger \bS_2^\top\bA$ takes $\tilde{O}\paren*{\nnz(\bA)+mk + lk^2/\e^2 + nlk/\e^2}$ time.    
\end{proof}

\subsection{Proof Sketch of Theorem \ref{t:main-technical}: Spectral Norm Bound} \label{subsec:specnormbd}

We now return to the proof of our main technical result. We start with the spectral norm $\|\bA-\bQ\bQ^\top\bA\|$ where $\bQ = \mathrm{orth}(\bB\bOmega)$ and $\bB=(\bA\bS(\bA\bS)^\top)^q\bA\bS$. Intuitively, the key difference between this and the classical power iteration is that we get a subspace which aligns more closely with the subspace spanned by the top $k$ singular vectors of $\tilde\bA :=\bA\bS$ instead of $\bA$. This means that if we project $\tilde\bA$ onto the column space of $\bQ$, we can get a good low-rank approximation of $\tilde\bA$. 

However, we need a low-rank approximation of $\bA$, and not $\tilde\bA$. So we use the regularized spectral approximation guarantee to show that the approximation error of the projection $\bP_{\bB\bOmega} := \bQ\bQ^\top$ on $\bA$ is bounded by the approximation error of the projection $\bP_{\bB\bOmega}$ on $\tilde\bA$. Formally, since $\tilde\bA$ is a $(\e, \lambda_1)$-regularized spectral approximation of $\bA$, we have,
     $\bA\bA^\top \preceq \frac{1}{1-\varepsilon}
(\tilde\bA \tilde\bA^\top +\varepsilon\lambda_1 \bI )$, which gives
\begin{align*}
    \paren*{\bI - \bP_{\bB \bOmega}}\bA\bA^\top\paren*{\bI - \bP_{\bB \bOmega}} &\preceq (1+2\e)
\paren*{\bI - \bP_{\bB \bOmega}}\big( \tilde\bA \tilde\bA ^\top +\varepsilon\lambda_1 \bI \big)\paren*{\bI - \bP_{\bB \bOmega}}
\end{align*}
where we use the fact that when $\e \le 1/2$, $\frac{1}{1-\varepsilon} \le 1+2\e$. With a careful application of Jensen's inequality, we use the above to show that,
\begin{equation*}
     \norm*{\paren*{\bI - \bP_{\bB \bOmega}}\bA}^2 \le \paren*{1+2\e} \paren*{ \norm*{\paren*{\bI - \bP_{\bB \bOmega}} \big(\tilde\bA \tilde\bA^\top\big)^{2q+1} \paren*{\bI - \bP_{\bB \bOmega}} }^\frac{1}{2q+1} + \varepsilon\lambda_1  }. \\
\end{equation*}
The norm under the $\frac1{2q+1}$ exponent simplifies to $\norm{(\bI - \bP_{\bB\bOmega})\bB}^2$, which can be bounded using Lemma~\ref{lem:projspecapprox}, obtaining $\norm*{(\bI - \bP_{\bB\bOmega})\bB}^2 \le 2\lambda_2$. See the complete proof in Appendix \ref{sec:spec_norm_proof}.

\subsection{Proof Sketch of Theorem \ref{t:main-technical}: Frobenius Norm Bound}
\label{subsec:forbnorm}

We now turn to bounding the Frobenius norm $\|\bA-\bQ\bQ^\top\bA\|_F^2$, which is essential for our low-rank factorization guarantee (Theorem \ref{thm:generalised_nystrom}). Curiously, while bounding the Frobenius norm error in low-rank approximation algorithms is typically more straightforward than the spectral norm, the presence of the power method makes the Frobenius norm argument more delicate in this case. 

Let $\bP_{\perp} := \bI - \bP_{\bB\bOmega} = \bI - \bQ\bQ^\top$. We bound the effect of $\bP_{\perp}$ on $\bA$ in two ways. By the Pythagorean theorem, we have $\norm{\bP_{\perp}\bA}^2_F = \norm{\bP_{\perp}\bA_r}^2_F + \norm{\bP_{\perp}\bA_{>r}}^2_F$, where $\bA_r$ represents the matrix formed by the top $r$ principal vectors of $\bA$ and $\bA_{>r} := \bA - \bA_r$. The second term, $\norm{\bP_{\perp}\bA_{>r}}^2_F$ can be at most $\norm{\bA_{>r}}^2_F = \norm{\bA - \bA_r}^2_F = \sum_{i>r} \sigma^2_{i}(\bA)$, so the main challenge is bounding the first term. 

Using the SVD of $\bA_r$, we get, $\norm{(\bI - \bP_{\bB \bOmega})\bA_r}^2_F = \sum_{i=1}^r \sigma_i^2(\bA)\, \bu_i^\top(\bI - \bP_{\bB \bOmega})\bu_i$, where $\bu_i$ are the left singular vectors of $\bA$. Following the proof of Lemma \ref{lem:projspecapprox}, since $\bB\bOmega$ is a $\lambda_2$-regularized $1/2$-spectral approximation of $\bB$, we have, $\bI - \bP_{\bB\bOmega} \preceq 2\lambda_2 (\bB\bB^\top + \lambda_2\bI)^{-1}$.
It follows that
\begin{align}
     \norm{(\bI - \bP_{\bB \bOmega})\bA_r}^2_F = \sum_{i=1}^r \sigma_i^2(\bA)\, \bu_i^\top(\bI - \bP_{\bB \bOmega})\bu_i \le \sum_{i=1}^r \sigma_i^2(\bA)\, 2 \lambda_2 \bu_i^\top \paren*{\bB \bB^\top + \lambda_2 \bI}^{-1} \bu_i. \label{eq:bspecapprox}   
\end{align}

However, we are faced with the mismatch that $\bu_i$ are the left singular vectors of $\bA$, but we have them acting on $\bB$, a power of $\tilde\bA=\bA \bS$, instead. Ideally we want, $\bu_i^\top \paren*{\bB \bB^\top + \lambda_2 \bI}^{-1} \bu_i \propto 1/\sigma_i^2(\bA)$ to kill the contribution of the large singular values, but a direct expansion of $\bu_i^\top \paren*{\bB \bB^\top + \lambda_2 \bI}^{-1} \bu_i$ gives us many $\sigma_j^2(\bB) = \sigma_j\big((\tilde\bA\tilde\bA^\top)^{2q+1}\big)$ terms in the denominator.

To overcome this, we work with $\bBt := (\tilde\bA \tilde\bA^\top)^q \tilde\bA /{\theta^{2q+1}} $, a scaled version of $\bB$.  This does not change the algorithm since $\bP_{\bB\bOmega} = \bP_{\bBt\bOmega}$. Then, for the indices with $\sigma_j^2(\tilde\bA) \ge \theta^2$ we have,
\begin{align*}
     \sigma_j^2(\bBt) = \frac{\sigma_j\big((\tilde\bA\tilde\bA^\top)^{(2q+1)}\big)}{\theta^{2(2q+1)}} \ge \paren*{\frac{\sigma_j^2(\tilde\bA)}{\theta^{2}}}^{2q+1} \ge \frac{\sigma_j^2(\tilde\bA)}{\theta^2}.
\end{align*}
Thus, when working with $\bBt$ instead of $\bB$, we are able to replace the $\sigma_j^2(\bBt)$ terms that arise in \eqref{eq:bspecapprox} with $\sigma_j^2(\tilde\bA)$ when $\sigma_j^2(\tilde\bA)$ is large. This allows us to use the regularized spectral approximation guarantee for $\tilde \bA$ with  $\lambda_1' = \lambda_1 + \lambda_2^{1/2q+1}$ (where $\lambda_1'$ is chosen to simplify the resulting calculations) obtaining, 
\[ \sigma_i^2(\bA)\, \bu_i^\top(\bI - \bP_{\bBt \bOmega})\bu_i \le \sigma_i^2(\bA) \cdot \frac{C \lambda_1'}{\sigma_i^2(\bA) + \lambda_1'} = O(\lambda_1'),  \]
since a given $\lambda_1$-regularized $\e$-spectral approximation $\tilde\bA$ is also a $\lambda_1'$-regularized $\e$-spectral approximation for $\lambda_1' \ge \lambda_1$. Further details can be found in the full proof in Appendix \ref{sec:frob_norm_proof}.

\section{Experiments}
\label{s:experiments}

\begin{figure}[t]
    \centering
    \begin{minipage}{0.32\linewidth}
        \centering
        \includegraphics[width=\linewidth]{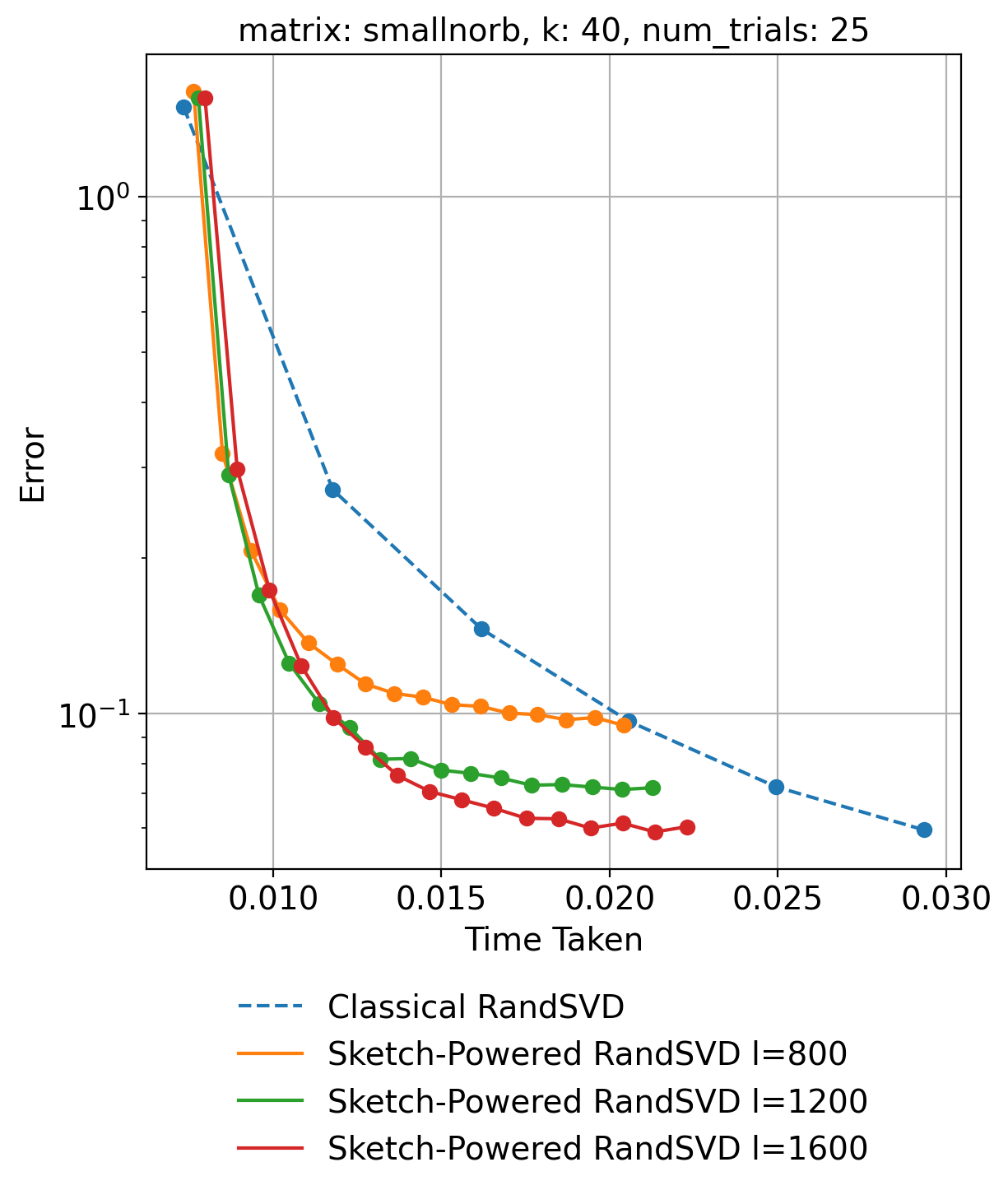}
    \end{minipage}\hfill
        \begin{minipage}{0.32\linewidth}
        \centering
        \includegraphics[width=\linewidth]{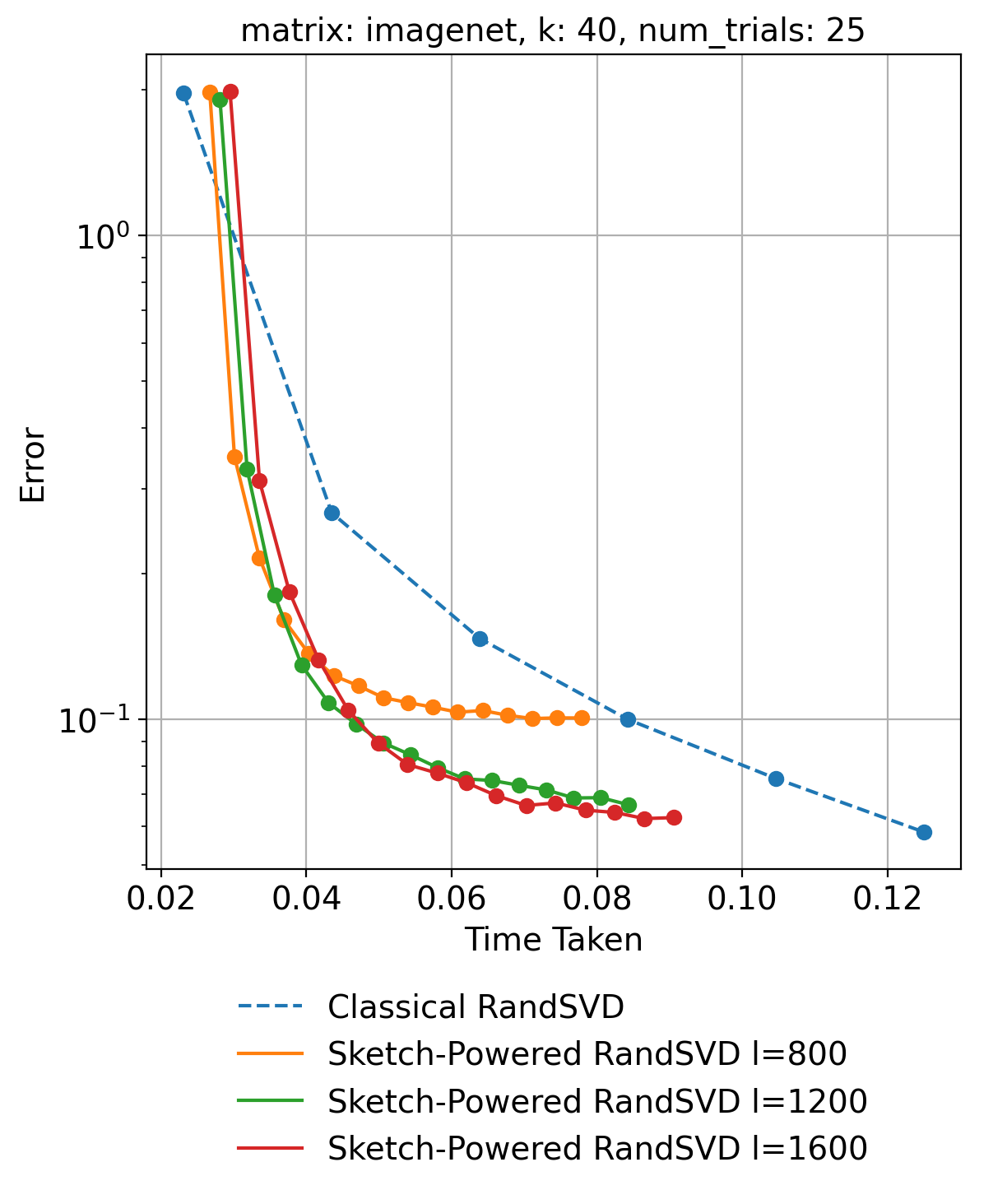}
    \end{minipage}\hfill
    \begin{minipage}{0.32\linewidth}
        \centering
        \includegraphics[width=\linewidth]{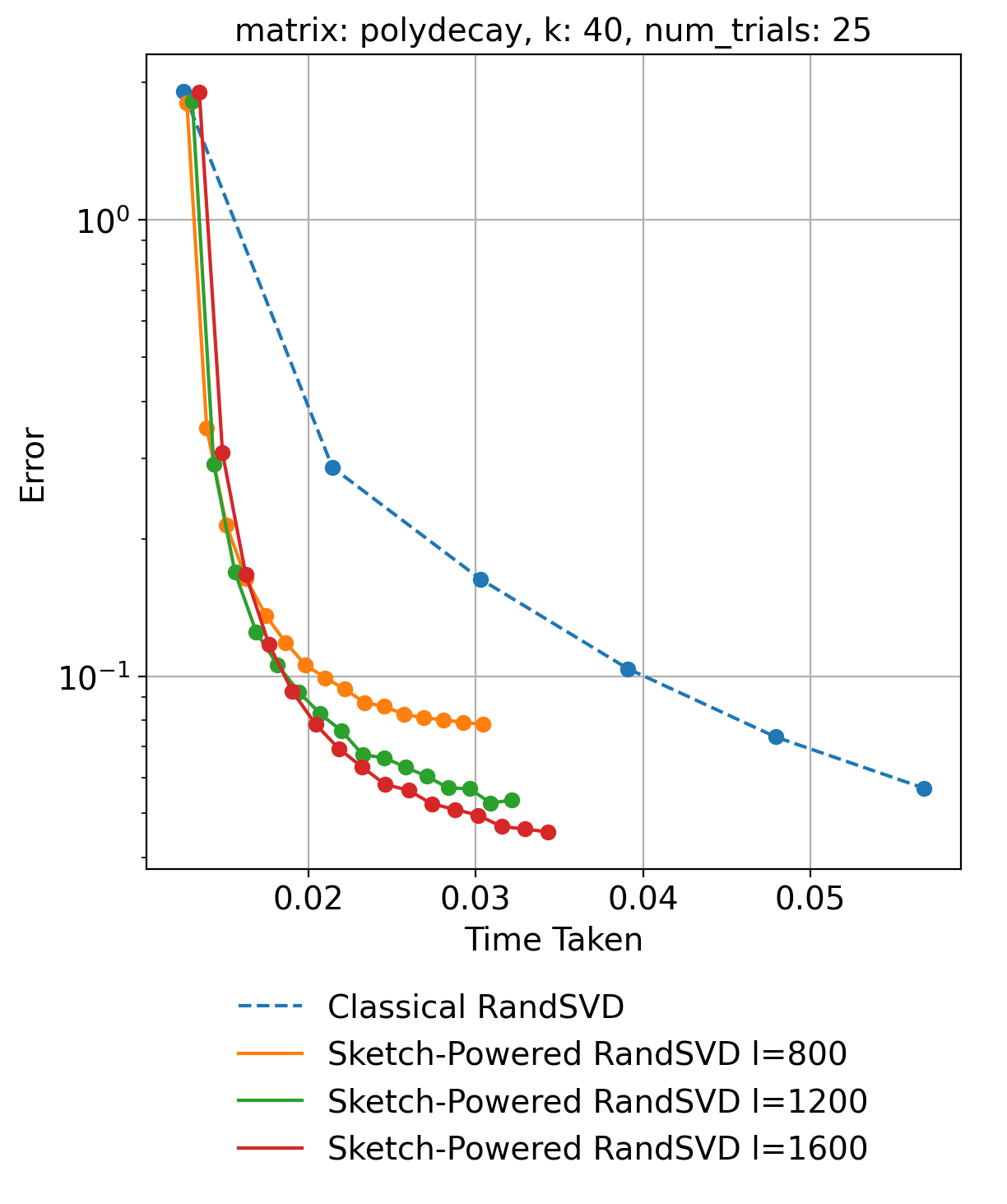}
    \end{minipage}
    \caption{Approximation error vs time for Classical Randomized SVD (RandSVD) and Sketch-Powered RandSVD for (from left to right) \texttt{smallnorb}, \texttt{ImageNet}, and \texttt{polydecay} matrices.}
    \label{fig:experiments}
\end{figure}

To validate our theory, we compare the numerical performance of sketched power method to its classical counterpart on three large data matrices. In this section, we evaluate a sketch-powered version of Randomized SVD, as it is arguably the most common application of the power~method \cite{halko2011finding} (see Appendix \ref{app:experiments} for more experiments). Concretely, we use Algorithm \ref{alg:column_sketched_iteration} to construct the range finder $\bQ$, and then use that to compute an approximate SVD, $\bA\approx \bU\bSigma\bV^\top\!=\bQ\bQ^\top\bA$, as described in Section \ref{s:main}. Focusing on the regime where $\e = \Theta(1)$, for simplicity we let $r_1 = l$ and $r_2 = k$.

The matrices we consider are: (1) \texttt{smallnorb} \cite{smallnorb} (size $24300 \times 9216$), (2) \texttt{ImageNet} \cite{imagenet} (size $20800 \times 50176$ matrix containing $50176$ red channel pixel values of $20800$  images resized to pixel size $224 \times 224$) and (3) \texttt{polydecay}, a synthetic matrix  of size $30000 \times 15000$ generated by multiplying on both sides a diagonal matrix with polynomially decaying values $30000/i$ by random orthogonal matrices  of the appropriate dimension. We let the sketching matrix $\bS\in\R^{n\times l}$ in Algorithm \ref{alg:column_sketched_iteration} be the CountSketch matrix \cite{countsketch,clarkson2017low} with one non-zero per row, and we use a Gaussian matrix $\bOmega\in\R^{l\times k}$ for the power method. The implementation was done using an A100 GPU in Google Colab and our code is available on the repository - \url{https://github.com/shabarishch/SketchedPowerMethod}.

The results of the comparison for target rank $k=40$ are in Figure \ref{fig:experiments}. For each iteration $q=0,1,...$ we plot the average relative error $\norm{\bA - \bU\bSigma\bV^\top}/\sigma_{k+1}(\bA)- 1$ for that iterate averaged over 25 trials versus the average time taken for that iterate. We vary the sketch size among $l\in\{800, 1200, 1600\}$.

Since the sketched power iterations are faster, we plot errors for 15 iterations against only 5 iterations for the classical power method. While the approximation error of the Sketch-Powered RandSVD does eventually start to plateau, the method reaches a moderately small relative error (say, $0.1$) much sooner than the Classical RandSVD does. Such moderately small error is sufficient for many machine learning applications \cite{gittens2013revisiting}, and the faster iteration time of the sketched power method more than makes up for the added overhead cost of forming the sketch $\bA\bS$. As expected, we observe that there is some trade-off between the iteration speed and error plateau when varying the sketch size $l$.

In Appendix \ref{app:experiments}, we also provide a numerical evaluation of the performance of Sketch-Powered Low-Rank Factorization (Algorithm \ref{alg:generalised_nystrom}). Here, we observe similar computational improvements from using sketched power method as in the case of RandSVD. 

\vspace{-0.1cm}
\section{Conclusions}
\vspace{-0.1cm}
We develop an algorithmic and theoretical framework for accelerating the power method using fast sketching. We use this framework to design and analyze efficient algorithms for singular value decomposition, low-rank factorization, and Nystr\"om approximation. We demonstrate the effectiveness of this approach on several benchmark data matrices.

\subsection*{Acknowledgments}
Partially supported by DMS 2054408, CCF 2338655, and a Google ML and Systems Junior Faculty Award. This work was done in part while MD was visiting the Simons Institute for the Theory of Computing. SC acknowledges support from the Allen Shields Memorial Fellowship and the Rackham One-Term Dissertation Fellowship at the University of Michigan. The authors thank Anil Damle, Raphael Meyer, Cameron Musco, Christopher Musco, and Mark Rudelson for helpful conversations.

\printbibliography


\appendix

\section{Related Work}
\label{a:related-work}

Early work in numerical linear algebra focused on using the power method for estimating the largest singular vector of a matrix \cite{kuczynski1992estimating}. The modern analysis for using the power method to construct the range finder and Randomized SVD was developed by \cite{rokhlin2010randomized}, and then extended and popularized by \cite{halko2011finding}, among others \cite{gu2015subspace,tropp2023randomized}. In this context, the power method is also often referred to as subspace iteration or power iteration. When seeking high-accuracy low-rank approximations, i.e., with error $(1+\e)\sigma_{k+1}(\bA)$ for $\e\ll 1$, the power method can be accelerated using Krylov subspaces \cite{musco2015randomized,bakshi2022low,chen2026does,meyer2024unreasonable}, which requires fewer iterations but has a higher per-iteration cost. In this work, we focus on computationally constrained environments where even the per-iteration cost of the classical power method is a bottleneck, and so we are primarily interested in the $\e=\Omega(1)$ regime.

Several prior works have considered approximate versions of the power method. In particular, \cite{hardt2014noisy} study a noisy version of the power method, where each application of $\bA$ is accompanied by some additive noise, a model that arises in applications to matrix completion \cite{jain2013low}, streaming algorithms \cite{mitliagkas2013memory}, and differential privacy \cite{hardt2013beyond}. While \cite{hardt2014noisy} introduce constraints on the additive noise that scale with the spectral gap of $\bA$, in our work, we are dealing with a very specific type of noise that comes from sketching, and are able to avoid dependence on any spectral gaps.

In a nearly concurrent work, \cite{chang2026improving} consider a sketched variant of the power method which is very similar to \eqref{eq:sketched-power}. However, even though they recommend the use of fast sketching matrices in practice, their theoretical analysis crucially requires the sketching matrices to be Gaussian, thereby limiting any computational benefits of the approach other than in single-pass streaming models. Our theoretical framework can be viewed as avoiding this key limitation through the use of the regularized spectral approximation argument, thus addressing a natural question arising from \cite{chang2026improving}.

Fast sketching has been extensively used more broadly in the area of Randomized Numerical Linear Algebra (RandNLA, \cite{woodruff2014sketching,drineas2016randnla,martinsson2020randomized,murray2023randomized,derezinski2024recent}), as well as specifically in the context of low-rank approximation, e.g., \cite{sarlos2006improved,drineas2006fast,clarkson2017low,halko2011finding,nakatsukasa2020fast}. Particularly relevant to our analysis is the line of works that has relied on the regularized spectral approximation property of sketching. This notion was first used for analyzing sub-sampling based methods in the context of Nystr\"om approximation \cite{alaoui2015fast,musco2017recursive,rudi2018fast,garg2025faster} and psd low-rank approximation \cite{musco2017sublinear}. In contrast, for sparse and structured sketching matrices such as those considered in this work, a more common approach has been to rely on the oblivious subspace embedding (OSE) property, e.g., \cite{sarlos2006improved,woodruff2014sketching,cohen_nelson_woodruff}. In fact, one can relatively straightforwardly pass from a slightly more general \emph{OSE moments} property \cite{cohen_nelson_woodruff} to regularized spectral approximation, e.g., as described by \cite{derezinskisidford}.
\section{Proof of Theorem \ref{t:main-technical}: Spectral Norm Bound }
\label{sec:spec_norm_proof}

Since $\bA \bS$ is an $\lambda_1$-regularized $\e$-spectral approximation of $\bA$, we have,
\begin{equation}\label{eq:column_stage1approx}
    (1-\varepsilon)\paren*{ \bA\bA^\top + \lambda_1\bI } \preceq \bA \bS (\bA \bS)^\top + \lambda_1\bI \preceq (1+\varepsilon)\paren*{ \bA\bA^\top + \lambda_1\bI }
\end{equation}
which gives,
\begin{equation*}
    \bA\bA^\top \preceq \frac{1}{1-\varepsilon}
\paren*{ \bA \bS (\bA \bS)^\top +\varepsilon\lambda_1 \bI }
\end{equation*}

Let $\bP$ be any orthogonal projection. Then, 
\begin{equation*}
    \bP\bA\bA^\top\bP \preceq \frac{1}{1-\varepsilon}
\bP\paren*{ \bA \bS (\bA \bS)^\top +\varepsilon\lambda_1 \bI }\bP
\end{equation*}
Let $\bx$ be a unit vector at which $\bP\bA\bA^\top\bP$ attains its norm. Then,
\begin{equation*}
     \norm*{\bP\bA}^2 \le \frac{1}{1-\varepsilon}
\bx^\top \bP\paren*{ \bA \bS (\bA \bS)^\top +\varepsilon\lambda_1 \bI }\bP\bx
\end{equation*}

Let $\bA \bS (\bA \bS)^\top +\varepsilon\lambda_1 \bI = \bU \paren*{\bD + \varepsilon\lambda_1\bI} \bU^\top $ be the eigendecomposition. Then,
\begin{align*}
    \bx^\top \bP\paren*{ \bA \bS (\bA \bS)^\top +\varepsilon\lambda_1 \bI }\bP\bx &= \bx^\top \bP\bU \paren*{ \bD +\varepsilon\lambda_1 \bI }\bU^\top\bP\bx \\
    &= \sum_{i=1}^m \paren*{\sigma_i(\bA \bS)^2 + \varepsilon\lambda_1} (\bU^\top\bP\bx)_i^2
\end{align*}
and,
\begin{align*}
     \norm*{\bP\bA}^{2(2q+1)} &\le \frac{1}{\paren{1-\varepsilon}^{2q+1}}
\paren*{ \sum_{i=1}^m \paren*{\sigma_i(\bA \bS)^2 + \varepsilon\lambda_1} (\bU^\top\bP\bx)_i^2 }^{2q+1} \\
&= \paren*{\frac{\norm{\bU^\top \bP \bx}_2^2}{1-\varepsilon}}^{2q+1}
\paren*{ \sum_{i=1}^m \paren*{\sigma_i(\bA \bS)^2 + \varepsilon\lambda_1} \frac{(\bU^\top\bP\bx)_i^2}{\norm{\bU^\top \bP \bx}_2^2} }^{2q+1} \\
&\le \paren*{\frac{\norm{\bU^\top \bP \bx}_2^2}{1-\varepsilon}}^{2q+1}
\paren*{ \sum_{i=1}^m \paren*{\sigma_i(\bA \bS)^2 + \varepsilon\lambda_1}^{2q+1} \frac{(\bU^\top\bP\bx)_i^2}{\norm{\bU^\top \bP \bx}_2^2} } \quad \text{ (Jensen)}\\
&\le \paren*{\frac{1}{1-\varepsilon}}^{2q+1} \norm{\bU^\top \bP \bx}_2^{2q}
\paren*{ \sum_{i=1}^m \paren*{\sigma_i(\bA \bS)^2 + \varepsilon\lambda_1}^{2q+1} (\bU^\top\bP\bx)_i^2  } \\
\implies \norm*{\bP\bA}^{2} &\le \paren*{\frac{1}{1-\varepsilon}} 
\paren*{ \sum_{i=1}^m \paren*{\sigma_i(\bA \bS)^2 + \varepsilon\lambda_1}^{2q+1} (\bU^\top\bP\bx)_i^2  }^\frac{1}{2q+1}\\
&\le \paren*{\frac{1}{1-\varepsilon}} \paren*{\paren*{ \sum_{i=1}^m \sigma_i(\bA \bS)^{2(2q+1)}(\bU^\top\bP\bx)_i^2  }^\frac{1}{2q+1} + \paren*{ \sum_{i=1}^m (\varepsilon\lambda_1)^{2q+1}(\bU^\top\bP\bx)_i^2  }^\frac{1}{2q+1}} \\
&\le \paren*{\frac{1}{1-\varepsilon}} \paren*{ \norm*{\bP \paren*{\bA \bS (\bA \bS)^\top}^{2q+1} \bP }^\frac{1}{2q+1} + \varepsilon\lambda_1  } \\
\end{align*}

where the second inequality after taking the $(2q+1)\textsuperscript{th}$ root follows from Minkowski's inequality (triangle inequality for $L_q$ norms). 

Now, $\norm*{\bP \paren*{\bA \bS (\bA \bS)^\top}^{2q+1} \bP } = \norm*{ \bP \paren*{  \bA \bS (\bA \bS)^\top}^q  \bA \bS }^2$. Thus,
\begin{align*}
    \norm*{\bP\bA}^2 \le \paren*{\frac{1}{1-\varepsilon}} \paren*{ \norm*{ \bP \paren*{  \bA \bS (\bA \bS)^\top}^q  \bA \bS }^{\frac{2}{(2q+1)}} + \varepsilon\lambda_1  }
\end{align*}

 Let $\bB = \paren*{\bA \bS (\bA \bS)^\top}^q \bA \bS$. Consider the secondary sketch $\bB \bOmega$ and let $\bP = \bI - \bP_{\bB\bOmega}$, where $\bP_{\bB\bOmega}$ is the orthogonal projection onto the column span of $\bB\bOmega$. Then,

 \begin{align*}
    \norm*{(\bI - \bP_{\bB\bOmega})\bA}^2 \le \paren*{\frac{1}{1-\varepsilon}} \paren*{ \norm*{ (\bI - \bP_{\bB\bOmega})\bB }^{\frac{2}{(2q+1)}}  + \varepsilon\lambda_1  }
\end{align*}

If $\bB \bOmega$ is a $\lambda_2$-regularized $1/2$-spectral approximation of $\bB$, by Lemma \ref{lem:projspecapprox}, $\norm*{(\bI - \bP_{\bB\bOmega})\bB}^2 \le 2\lambda_2$. So we get, when $\e \le 1/2$,
 \begin{align*}
    \norm*{(\bI - \bP_{\bB\bOmega})\bA}^2 \le \paren*{1+2\e} \paren*{ \paren*{ 2\lambda_2 }^{\frac{1}{(2q+1)}}  + \varepsilon\lambda_1  }
\end{align*}

\section{Proof of Theorem \ref{t:main-technical}: Frobenius Norm Bound }
\label{sec:frob_norm_proof}

    Let $\bP = \bQ\bQ^\top$. We want to estimate $\norm{(\bI-\bP)\bA}^2_F$. Let $\bA = \bA_r + \bA_{>r}$ be the split given by the SVD of $\bA$, where $\bA_r = \sum_{i=1}^{r} \sigma_i(\bA) \bu_i \bv_i^\top$ retains the top $r$ principal directions and $\bA_{>r} = \sum_{i=r+1}^{p} \sigma_i(\bA) \bu_i \bv_i^\top$ contains the remaining directions where $p$ is the rank of $\bA$. Since $\{\bv_i\}_{i=1}^r$ and $\{\bv_i\}_{i=r+1}^p$ are orthogonal sets of right singular vectors, the cross terms vanish and we have
    \begin{align*}
        \norm{(\bI - \bP)\bA}^2_F
        &= \norm{(\bI - \bP)(\bA_r + \bA_{>r})}^2_F \\
        &= \norm{(\bI-\bP)\bA_r}^2_F + 2\ip*{(\bI-\bP)\bA_r, (\bI-\bP)\bA_{>r}}_F + \norm{(\bI-\bP)\bA_{>r}}^2_F \\
        &= \norm{(\bI-\bP)\bA_r}^2_F + \norm{(\bI-\bP)\bA_{>r}}^2_F
    \end{align*}
    where the last equality uses $\ip*{(\bI-\bP)\bA_r, (\bI-\bP)\bA_{>r}}_F = \tr\paren*{(\bI-\bP)\bA_{>r}\bA_r^\top(\bI-\bP)} = 0$, since $\bA_{>r}\bA_r^\top = \sum_{i=1}^r \sum_{j=r+1}^p \sigma_i(\bA) \sigma_j(\bA) \bu_j \bv_j^\top \bv_i \bu_i^\top = 0$ by orthogonality of right singular vectors.

    For the first term, we expand using $\bA_r = \sum_{i=1}^r \sigma_i(\bA) \bu_i \bv_i^\top$:
    \begin{align*}
        \norm{(\bI-\bP)\bA_r}^2_F
        &= \tr\paren*{(\bI-\bP)\bA_r \bA_r^\top(\bI-\bP)} \\
        &= \tr\paren*{(\bI-\bP) \sum_{i=1}^r\sum_{j=1}^r \sigma_i(\bA)\sigma_j(\bA) \bu_i \bv_i^\top \bv_j \bu_j^\top (\bI-\bP)} \\
        &= \sum_{i=1}^r \sigma_i(\bA)^2\, \tr\paren*{(\bI-\bP)\bu_i\bu_i^\top(\bI-\bP)} \\
        &= \sum_{i=1}^r \sigma_i(\bA)^2\, \bu_i^\top(\bI-\bP)^2\bu_i = \sum_{i=1}^r \sigma_i(\bA)^2\, \bu_i^\top(\bI-\bP)\bu_i
    \end{align*}
    where the third equality uses orthonormality of the right singular vectors $\bv_i^\top \bv_j = \delta_{ij}$, and the last uses $(\bI-\bP)^2 = \bI-\bP$.

    For the second term, since $\bI - \bP$ is an orthogonal projection we have $(\bI-\bP)^2 = \bI - \bP \preceq \bI$, so
    \begin{align*}
        \norm{(\bI-\bP)\bA_{>r}}^2_F
        &= \tr\paren*{(\bI-\bP)\bA_{>r}\bA_{>r}^\top(\bI-\bP)} \\
        &= \tr\paren*{\bA_{>r}^\top(\bI-\bP)^2\bA_{>r}} \\
        &\le \tr\paren*{\bA_{>r}^\top\bA_{>r}} = \norm{\bA_{>r}}^2_F = \sum_{i=r+1}^{p} \sigma_i(\bA)^2.
    \end{align*}
Since we are given that, 
$\bA \bS$ is a $\lambda_1$-regularized $\e$-spectral approximation of $\bA$, we have,
\begin{equation*}
    (1-\varepsilon)\paren*{ \bA\bA^\top + \lambda_1\bI } \preceq \bA \bS (\bA \bS)^\top + \lambda_1\bI \preceq (1+\varepsilon)\paren*{ \bA\bA^\top + \lambda_1\bI }
\end{equation*}
However, observe that the above approximation still holds if $\lambda_1$ is replaced by $\lambda_1' \ge \lambda_1$. So letting $\lambda_1' = \lambda_1 + \lambda_2^{1/2q+1}$, we may assume
\begin{equation} \label{eq:lambda1primeapprox}
    (1-\varepsilon)\paren*{ \bA\bA^\top + \lambda_1'\bI } \preceq \bA \bS (\bA \bS)^\top + \lambda_1'\bI \preceq (1+\varepsilon)\paren*{ \bA\bA^\top + \lambda_1'\bI }
\end{equation}

Recall that $\bP = \bP_{\bB\bOmega}$, where $\bP_{\bB\bOmega}$ is orthogonal projection onto the column span of $\bB\bOmega$ and $\bB = \paren*{\bA\bS (\bA\bS)^\top}^q \bA\bS$. Instead of working with $\bB$, we will instead work with a scaled version of $\bB$ defined as a power of a scaled version of $\bA \bS$. Let $\bBt := \frac{1}{\theta^{2q+1}} \paren*{\bA\bS (\bA\bS)^\top}^q \bA\bS$ for $\theta=\sqrt{\lambda_1'}$. This is not a problem since $\bP_{\bB\bOmega} = \bP_{\bBt\bOmega}$. 

If $\bB\bOmega$ is a $\lambda_2$-regularized $1/2$-spectral approximation of $\bB$, then  $\bBt\bOmega$ is a $\lambda_2/\theta^{2(2q+1)}$-regularized $1/2$-spectral approximation of $\bBt$. Since $\theta^{2(2q+1)} = \lambda_1'^{2q+1} \ge \lambda_2$, we may assume that $\bBt\bOmega$ is a $1$-regularized $1/2$-spectral approximation of $\bBt$.

Thus, by the proof of Lemma \ref{lem:projspecapprox}, $\bI-\bP  \preceq 2(\bBt\bBt^\top + \bI)^{-1} $.
 
For any fixed $i$, this gives $\bu_i^\top(\bI-\bP)\bu_i \le 2\,\bu_i^\top(\bBt\bBt^\top + \bI)^{-1}\bu_i$. We now bound $\bu_i^\top(\bBt\bBt^\top + \bI)^{-1}\bu_i$, 
 \begin{align*}
     \bu_i^\top(\bBt\bBt^\top + \bI)^{-1}\bu_i
     &= \bu_i^\top \bW\paren*{\frac{1}{\theta^{2(2q+1)}}\bD^{2q+1} + \bI}^{-1}\bW^\top\bu_i \\
     &= \sum_{j:\sigma_j(\bA\bS)\ge \theta} \frac{(\bW^\top\bu_i)_j^2}{\paren*{\sigma_j(\bA\bS)/\theta}^{2(2q+1)} + 1} \\
     &\quad + \sum_{j:\sigma_j(\bA\bS)< \theta} \frac{(\bW^\top\bu_i)_j^2}{\paren*{\sigma_j(\bA\bS)/\theta}^{2(2q+1)} + 1},
     \\
     &\le \sum_{j:\sigma_j(\bA\bS)\ge \theta} \frac{(\bW^\top\bu_i)_j^2}{\paren*{\sigma_j(\bA\bS)/\theta}^{2}} + \sum_{j:\sigma_j(\bA\bS)< \theta} \frac{(\bW^\top\bu_i)_j^2}{1} \\
     &= \theta^2 \paren*{ \sum_{j:\sigma_j(\bA\bS)\ge \theta} \frac{(\bW^\top\bu_i)_j^2}{\sigma_j(\bA\bS)^{2}} + \sum_{j:\sigma_j(\bA\bS)< \theta} \frac{(\bW^\top\bu_i)_j^2}{\theta^2} } \\
     &\le \theta^2 \paren*{ \sum_{j:\sigma_j(\bA\bS)\ge \theta} \frac{(\bW^\top\bu_i)_j^2}{\sigma_j(\bA\bS)^{2}} + \sum_{j:\sigma_j(\bA\bS)< \theta} \frac{(\bW^\top\bu_i)_j^2}{\theta^2} }
 \end{align*}
 where $\bA\bS(\bA\bS)^\top = \bW\bD\bW^\top$ is the eigendecomposition with $\bD = \diag(\sigma_1(\bA\bS)^2, \ldots, \sigma_{\widehat{m}}(\bA\bS)^2)$ (for $ \widehat{m} := \min\{ m, r_1 \}$), so $\bBt\bBt^\top = \theta^{-2(2q+1)}\bW\bD^{2q+1}\bW^\top$. Also, by \eqref{eq:lambda1primeapprox}, 
 \[
     \paren*{\bA\bS(\bA\bS)^\top + \lambda_1'\bI}^{-1} \preceq \frac{1}{1-\varepsilon}\paren*{\bA\bA^\top + \lambda_1'\bI}^{-1}.
 \]
 Hence, for each left singular vector $\bu_i$ of $\bA$,
 \[
     \bu_i^\top \paren*{\bA\bS(\bA\bS)^\top + \lambda_1'\bI}^{-1} \bu_i
     \le \frac{1}{1-\varepsilon}\,\bu_i^\top \paren*{\bA\bA^\top + \lambda_1'\bI}^{-1}\bu_i
     = \frac{1}{(1-\varepsilon)(\sigma_i(\bA)^2 + \lambda_1')}
 \]
 Equivalently, using the eigendecomposition of $\bA\bS(\bA\bS)^\top$,
 \begin{align*}
     \bu_i^\top \paren*{\bA\bS(\bA\bS)^\top + \lambda_1'\bI}^{-1} \bu_i
     &= \bu_i^\top \bW(\bD + \lambda_1'\bI)^{-1}\bW^\top \bu_i \\
     &= \sum_{j:\sigma_j(\bA\bS)^2 \ge \lambda_1'} \frac{(\bW^\top\bu_i)_j^2}{\sigma_j(\bA\bS)^2 + \lambda_1'}
     + \sum_{j:\sigma_j(\bA\bS)^2 < \lambda_1'} \frac{(\bW^\top\bu_i)_j^2}{\sigma_j(\bA\bS)^2 + \lambda_1'}
     \\
     &\ge \sum_{j:\sigma_j(\bA\bS)^2 \ge \lambda_1'} \frac{(\bW^\top\bu_i)_j^2}{2\sigma_j(\bA\bS)^2}
     + \sum_{j:\sigma_j(\bA\bS)^2 < \lambda_1'} \frac{(\bW^\top\bu_i)_j^2}{2\lambda_1'} \\
     2 \cdot \bu_i^\top \paren*{\bA\bS(\bA\bS)^\top + \lambda_1'\bI}^{-1} \bu_i &\ge \sum_{j:\sigma_j(\bA\bS)^2 \ge \lambda_1'} \frac{(\bW^\top\bu_i)_j^2}{\sigma_j(\bA\bS)^2} + \sum_{j:\sigma_j(\bA\bS)^2 < \lambda_1'} \frac{(\bW^\top\bu_i)_j^2}{\lambda_1'} \\
 \end{align*}

 Thus, we have,
 \begin{align*}
    \sum_{j:\sigma_j(\bA\bS)^2 \ge \lambda_1'} \frac{(\bW^\top\bu_i)_j^2}{\sigma_j(\bA\bS)^2} + \sum_{j:\sigma_j(\bA\bS)^2 < \lambda_1'} \frac{(\bW^\top\bu_i)_j^2}{\lambda_1'} \le \frac{2}{(1-\varepsilon)(\sigma_i(\bA)^2 + \lambda_1')}
 \end{align*}
 Since $\theta^2 = \lambda_1'$, we get,
 \begin{align*}
    \bu_i^\top(\bBt\bBt^\top + \bI)^{-1}\bu_i \le \frac{2\theta^2}{(1-\varepsilon)(\sigma_i(\bA)^2 + \lambda_1')}
 \end{align*}
 Thus,
 \begin{align*}
    \norm{(\bI-\bP)\bA_r}^2_F &= \sum_{i=1}^r \sigma_i(\bA)^2\, \bu_i^\top(\bI-\bP)\bu_i \\
    &\le \sum_{i=1}^r \sigma_i(\bA)^2\, 2\,\bu_i^\top(\bBt\bBt^\top + \bI)^{-1}\bu_i \\
    &\le \sum_{i=1}^r \sigma_i(\bA)^2\, 2 \cdot \frac{2\theta^2}{(1-\varepsilon)(\sigma_i(\bA)^2 + \lambda_1')} \\
    &\le \frac{4r \theta^2  }{(1-\varepsilon)}
 \end{align*}

 Combining with the bound for $\norm{(\bI-\bP)\bA_{>r}}^2_F$, we get,
 \begin{align*}
        \norm{(\bI - \bP)\bA}^2_F
        &= \norm{(\bI-\bP)\bA_r}^2_F + \norm{(\bI-\bP)\bA_{>r}}^2_F \\
        &\le \sum_{i>r} \sigma^2_{i}(\bA) + {8r(\lambda_1 + \lambda_2^{1/2q+1}) } 
\end{align*}
when $\e < 1/2$.

\section{Additional Experiments} \label{app:experiments}

In this section, we compare the numerical performance of Sketch-Powered Low-Rank Factorization (Algorithm \ref{alg:generalised_nystrom}) against the version of Algorithm \ref{alg:generalised_nystrom} that does iterations without sketching, i.e. when $\bS_1 = \bI$. This baseline (denoted simply as ``Low-Rank Factorization" in our plots) precisely matches the so-called Generalized Nystr\"om algorithm as described, e.g., by \cite{nakatsukasa2020fast}, except this algorithm is stated there without power iteration (i.e., $q=0$). Generalized Nystr\"om is designed to avoid the expensive matrix multiplication step $\bQ^\top\bA$ that occurs in Randomized SVD. Because of this, prior work would have little reason to incorporate full power method into this procedure, as that would introduce even more expensive matrix multiplications, erasing most of this gain. Nevertheless, for completeness, in our experiments we show the performance of this baseline for a range of $q\geq 0$.

We consider the same target rank $k=40$ and the same matrices \texttt{smallnorb}, \texttt{ImageNet} and \texttt{polydecay} as described in Section \ref{s:experiments} and, as earlier, plot the average relative error with respect to the average time taken for a particular iterate in Figure \ref{fig:gn-experiments} for $l \in\{800, 1200, 1600\}$. Note that in this experiment the baseline algorithm also depends on the parameter $r_1=l$, which corresponds to the size of the sketching matrix $\bS_2$.

Similar to what we observed in Figure \ref{fig:experiments} for Sketch-Powered Randomized SVD, Algorithm \ref{alg:generalised_nystrom} reaches a relative error $\approx\!0.1$ faster than the baseline algorithm, and then at a certain point the error starts to plateau. Increasing the intermediate sketch size $l$ decreases the level at which the relative error plateaus, as predicted by Theorem \ref{thm:generalised_nystrom}.

\begin{figure}[t]
    \centering
    \begin{minipage}{0.32\linewidth}
        \centering
        \includegraphics[width=\linewidth]{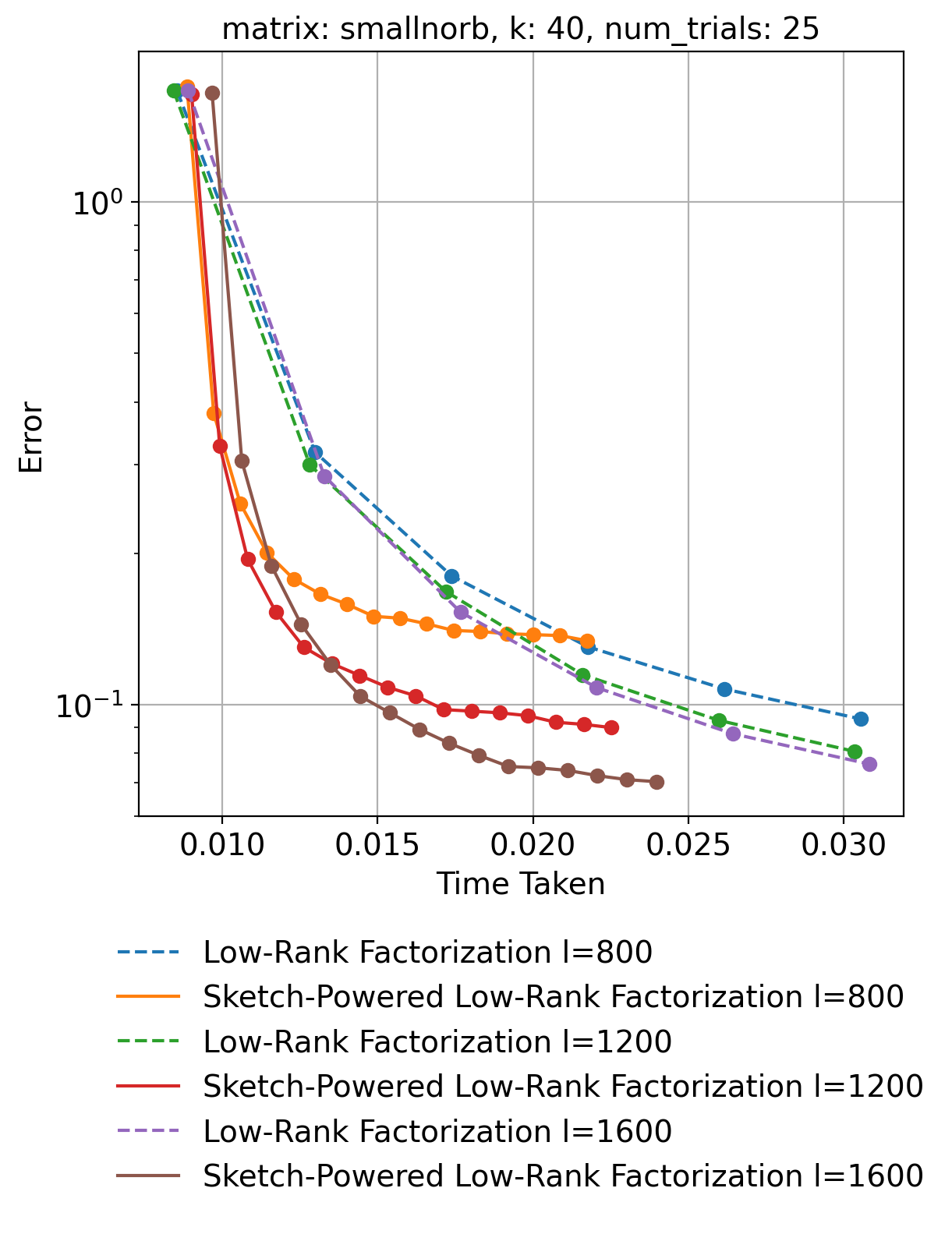}
    \end{minipage}
    \begin{minipage}{0.32\linewidth}
        \centering
        \includegraphics[width=\linewidth]{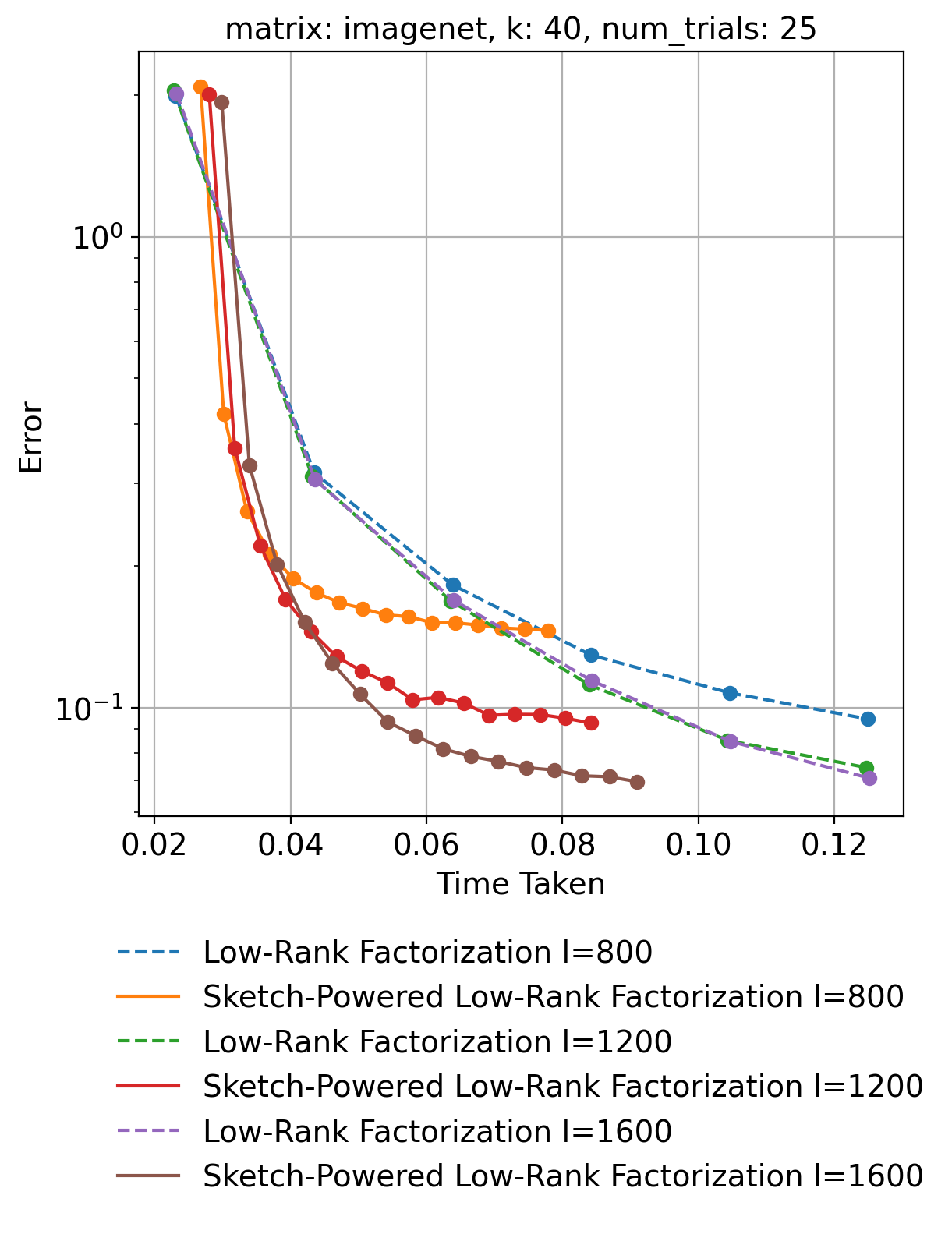}
    \end{minipage}\hfill
        \begin{minipage}{0.32\linewidth}
        \centering
        \includegraphics[width=\linewidth]{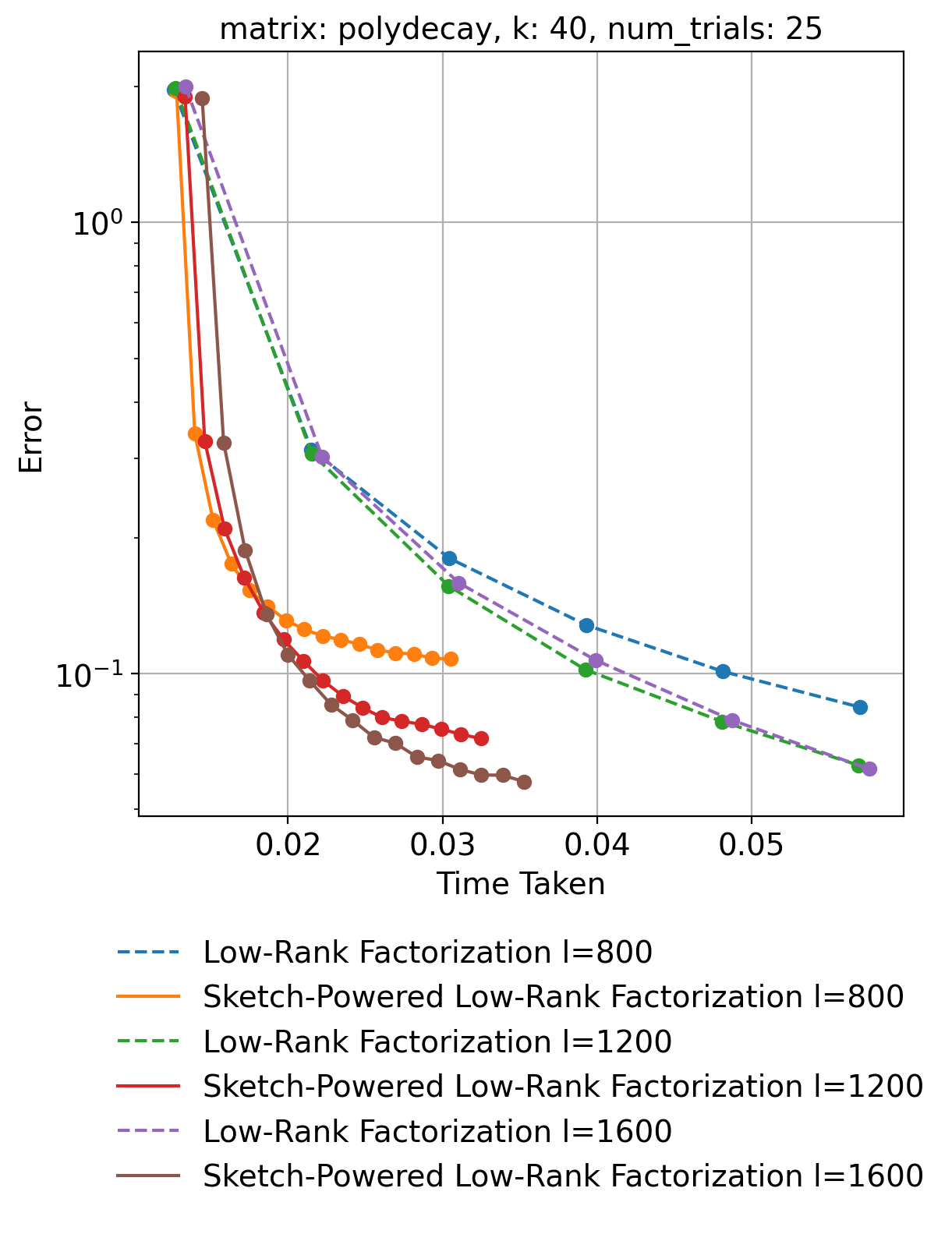}
    \end{minipage}\hfill
    \caption{Approximation error vs time for Sketch-Powered Low-Rank Factorization (Algorithm \ref{alg:generalised_nystrom}) compared with the same algorithm used without sketching the power method, for (from left to right) \texttt{smallnorb}, \texttt{ImageNet}, and \texttt{polydecay} matrices.}
    \label{fig:gn-experiments}
\end{figure}

\section{Miscellaneous Lemmas}

\subsection{Error Term Estimate}
\begin{lemma}\label{lem:lambda2bound}
    For $\bA \in \R^{m \times n}$, let $\bA\bS$ be a $\lambda_1$-regularized $\e$-spectral approximation of $\bA$ for some $\e\in[0,1/2]$ and let $\widehat{m} := \rank(\bA\bS)$. Then, for $k>0$ and $q \ge \log(2\widehat{m})/2\varepsilon$,
    \[ \paren*{ \frac{2}{k} \sum_{i>k} \sigma_i(\bA \bS)^{2(2q+1)} }^\frac{1}{2q+1} \le (1+4\e)\sigma_{k+1}(\bA)^2 + 2\lambda_1\e \]
\end{lemma}

\begin{proof}
By the spectral approximation guarantee,
\begin{align*}
    (1-\varepsilon)\paren*{ \bA\bA^\top + \lambda_1\bI } \preceq \bA \bS (\bA \bS)^\top &+ \lambda_1\bI \preceq (1+\varepsilon)\paren*{ \bA\bA^\top + \lambda_1\bI } \\
    (1-\varepsilon)\paren*{ \sigma_i(\bA)^2 + \lambda_1 } \le \sigma_i(\bA \bS)^2 &+ \lambda_1 \le (1+\varepsilon)\paren*{ \sigma_i(\bA)^2 + \lambda_1 } 
\end{align*}
So, for $\widehat{m} = \rank(\bA\bS)$,
\begin{align*}
    \paren*{\sum_{i>k} \sigma_i(\bA \bS)^{2(2q+1)} }^\frac{1}{2q+1} &\le \paren*{\sum_{i=k+1}^{\widehat{m}} \paren*{(1+\varepsilon) \sigma_i(\bA)^2  + \lambda_1\varepsilon }^{(2q+1)} }^\frac{1}{2q+1} \\
    &\le (1+\varepsilon) \paren*{\sum_{i=k+1}^{\widehat{m}} \sigma_i(\bA)^{2(2q+1)} }^\frac{1}{2q+1} + \lambda_1 \e \widehat{m}^{{1}/{2q+1}} \\
    \paren*{ \frac{2}{k} \sum_{i>k} \sigma_i(\bA \bS)^{2(2q+1)} }^\frac{1}{2q+1} &\le (1+\varepsilon) \paren*{ \frac{2}{k} \sum_{i=k+1}^{\widehat{m}} \sigma_i(\bA)^{2(2q+1)} }^\frac{1}{2q+1} + \lambda_1 \e \paren*{\frac{2\widehat{m}}{k}}^{{1}/{2q+1}}
\end{align*}
  Then, for $q = \log(2\widehat{m})/2\varepsilon$, $\paren*{{\widehat{2m}}/{k}}^{{1}/{2q+1}} \le (2\widehat{m})^{(\e/\log(2\widehat{m}))} \le e^\e \le 1+2\e \le 2$ (since $\e \le 1/2$). 
Also,
\begin{align*}
    \paren*{ \frac{2}{k} \sum_{i=k+1}^{\widehat{m}} \sigma_i( \bA)^{2(2q+1)} }^\frac{1}{2q+1}
    &\le \sigma_{k+1}(\bA)^2 \paren*{ \frac{2}{k} \sum_{i=k+1}^{\widehat{m}} \frac{\sigma_i( \bA)^{2(2q+1)}}{\sigma_{k+1}( \bA)^{2(2q+1)}}}^\frac{1}{2q+1} \\
    &\le \sigma_{k+1}(\bA)^2 (2\widehat{m})^\frac{1}{2q} \\
    &\le (1+2\e)\sigma_{k+1}(\bA)^2
\end{align*}
where we used that $\sigma_i(\bA) \le \sigma_{k+1}(\bA)$ for all $i > k$. Thus,
\begin{align*}
    \paren*{ 2\lambda_2 }^{\frac{1}{(2q+1)}} &= \paren*{ \frac{2}{k} \sum_{i>k} \sigma_i(\bA \bS)^{2(2q+1)} }^\frac{1}{2q+1} \\
    &\le (1+\e)(1+2\e)\sigma_{k+1}(\bA)^2 + 2\lambda_1\e \\
    &\le (1+4\e)\sigma_{k+1}(\bA)^2 + 2\lambda_1\e
\end{align*}
for $\e<1/2$. 
\end{proof}

\subsection{Proof of Lemma \ref{lem:specapprox} and Lemma \ref{lem:genreg}} 
\label{subsec:ammproofs}
    Lemma \ref{lem:specapprox} and Lemma \ref{lem:genreg} rely on the oblivious subspace embedding (OSE) moments  property defined in \cite{cohen_nelson_woodruff},
    
    \begin{definition}[OSE moments restated for right sketches, \cite{cohen_nelson_woodruff}, Definition 4] \label{def:osemoments}
        A distribution $\mathcal{D}$ over sketching matrices $\bS\in\R^{n\times r}$ has $(\e,\delta,d,p)$-OSE moments if, for every matrix $\bU\in\R^{n\times d}$ with orthonormal columns,
        \[
            \E_{\bS\sim\mathcal{D}} \norm*{\bU^\top\bS\bS^\top\bU-\bI_d}^{p} < \e^{p}\delta.
        \]
    \end{definition}

    By the discussion in Sections 2.1.1-2.1.3 of \cite{cohen_nelson_woodruff}, we note that the sketching matrices listed in Lemma~\ref{lem:specapprox} satisfy the $(\e,\delta,2k,p)$-OSE moment property for some $p\ge 2$ 
    (Going from $k$ to $2k$ only influences the sketch size lower bounds in Lemma~\ref{lem:specapprox} by constants, which get absorbed in the $O(\cdot)$ notation). We then have the following theorem for the sketching matrices listed in Lemma~\ref{lem:specapprox}, 

    \begin{theorem}[Definition 2, Theorem~6, \cite{cohen_nelson_woodruff}] \label{thm:amm}
        If $\bS\in\R^{n\times r}$ is one of the matrices listed in Lemma~\ref{lem:specapprox}, then $\bS$ satisfies the $(\e,\delta,2k,p)$-OSE moments property for some $p\ge 2$. Thus, with probability at least $1-\delta$, $\bS^\top$ satisfies the $(k,\e)$-approximate spectral norm matrix multiplication (AMM) property, where we say that $\bS^\top \in\R^{r\times n}$ satisfies the $(k,\e)$-AMM property if, for every pair of matrices $\bM,\mathbf{N}$ with $n$ rows,
        \[
            \norm*{(\bS^\top\bM)^\top(\bS^\top\mathbf{N})-\bM^\top\mathbf{N}}
            \le \e\sqrt{
            \paren*{\norm{\bM}^2+\frac{\norm{\bM}_F^2}{k}}
            \paren*{\norm{\mathbf{N}}^2+\frac{\norm{\mathbf{N}}_F^2}{k}}
            }.
        \]
    \end{theorem}

    Using this property, we can reproduce the proof of \cite[Lemma 12]{derezinskisidford} to get a proof of Lemma \ref{lem:specapprox},

    \begin{proof}[Proof of Lemma \ref{lem:specapprox}]
        Let $\lambda = \frac{1}{k}\sum_{i>k}\sigma_i^2(\bA)$ and define
        \[
            \bC := \big(\bA\bA^\top+\lambda\bI\big)^{-1/2}\bA.
        \]
        As in the proof of \cite[Lemma~12]{derezinskisidford}, it suffices to control the normalized sketching error:
        \begin{align*}
            (1-\e)\big(\bA\bA^\top+\lambda\bI\big)
            \preceq\;& \bA\bS\bS^\top\bA^\top+\lambda\bI
            \preceq (1+\e)\big(\bA\bA^\top+\lambda\bI\big) \\
            \iff\;&
            \big\|\bC\bS\bS^\top\bC^\top-\bC\bC^\top\big\|\le \e.
        \end{align*}
        By Theorem \ref{thm:amm}, after adjusting the constant hidden in the sketch size lower bound, $\bS^\top$ satisfies the $(k,\e/3)$-AMM property with probability at least $1-\delta$. Applying AMM with $\bM=\mathbf{N}=\bC^\top$ gives
        \begin{equation} \label{eq:multreduction}
                        \norm*{\bC\bS\bS^\top\bC^\top-\bC\bC^\top}
            \le \frac{\e}{3}\paren*{\norm{\bC}^2+\frac{\norm{\bC}_F^2}{k}}.
        \end{equation}
        Since $\norm{\bC}\le 1$ and
        \begin{align*}
            \norm{\bC}_F^2
            &= \sum_{i=1}^{\rank(\bA)}\frac{\sigma_i^2(\bA)}{\sigma_i^2(\bA)+\lambda} \\
            &\le k+\sum_{i>k}\frac{\sigma_i^2(\bA)}{\lambda}
            =2k,
        \end{align*}
        the right-hand side of \eqref{eq:multreduction} is at most $\e$. This proves the regularized spectral approximation guarantee. The stated sketch-size bounds follow from the OSE moment bounds in Sections 2.1.1-2.1.3 of \cite{cohen_nelson_woodruff}, with constant factors absorbed into the $O(\cdot)$ notation.
    \end{proof}

    For proving lemma \ref{lem:genreg}, we consider another theorem of \cite{cohen_nelson_woodruff} restated for right sketches, 

    \begin{theorem}[Theorem 7, \cite{cohen_nelson_woodruff}] \label{thm:genregcnw}
        Let $\bA\in\R^{n\times d},\bB\in\R^{n\times p}$ and $\bS\in\R^{n\times r}$, then for $\bXt :=(\bS^\top \bA)^\dagger \bS^\top\bB$, we have, 
            \[
        \norm{\bA\bXt-\bB}^2
        \le (1+\nu)\norm{\bP_{\bA}\bB-\bB}^2 + \frac{\nu}{k}\norm{\bP_{\bA}\bB-\bB}_F^2
    \] if, 
    \begin{enumerate}
        \item $\bS^\top$ satisfies the $(k,\sqrt{\nu/8})$-AMM property for $\bU_{\bA}$ and $\bP_{\bar{\bA}}\bB$, where $\bU_{\bA}$ is an orthonormal basis for the column space of $\bA$, $\bP_{\bA}$ is the orthogonal projection onto the column space of $\bA$,  and $\bP_{\bar{\bA}}:=\bI-\bP_{\bA}$; \label{cond:amm}
        \item $\bS^\top$ s a (1/2)-subspace embedding for the column space of $\bA$. \label{cond:se}
    \end{enumerate}
    \end{theorem}

    \begin{proof}[Proof of Lemma \ref{lem:genreg}]
        We may assume that $\bS$ has the $(\e/2\sqrt{2}, \delta',2k,p)$-OSE moments property by adjusting constants in the lower bound for $r$ in the statement of Lemma \ref{lem:specapprox}, which get absorbed in the $O(\cdot)$ notation and using the claims in Sections 2.1.1-2.1.3 of \cite{cohen_nelson_woodruff}.
        
        Then, by Theorem \ref{thm:amm}, if $\bS$ has the $(\e/2\sqrt{2}, \delta',2k,p)$-OSE moments property, then with probability at least $1-\delta'$, $\bS^\top$ satisfies the $(k, \e/2\sqrt{2})$ approximate spectral norm matrix multiplication property,  for $\bU_{\bA}$ and $\bP_{\bar{\bA}}\bB = (\bI - \bP_{\bA})\bB$.  

        The same AMM property, applied to $\bU_{\bA}$ (which has rank at most $k$, so $\norm{\bU_{\bA}}_F^2/k \le 1$) with itself, gives
        \begin{equation}\label{eq:1/2se}
            \norm*{(\bS^\top\bU_{\bA})^\top(\bS^\top\bU_{\bA})-\bI}\le \frac{\e}{2\sqrt2} \paren*{ \norm{\bU_{\bA}}^2 + \norm{\bU_{\bA}}_F^2/k } \le \frac12
        \end{equation}
        with probability at least $1-\delta'$ for $\e \le 1/2$. Equivalently, $\bS^\top$ is a $1/2$-subspace embedding for the column space of $\bA$.

        Thus we can apply Theorem \ref{thm:genregcnw}, with $\sqrt{\nu/8} = \e/2\sqrt2$, or $\nu = \e^2$ to get the claimed bound after letting $\delta' = \delta/2$ and taking a union bound for conditions \ref{cond:amm} and \ref{cond:se} of Theorem \ref{thm:genregcnw} to hold simultaneously and absorbing constant factors into the sketch size lower bounds in Lemma \ref{lem:specapprox}.
    \end{proof}

\newpage
\end{document}